\documentclass[10pt]{article}
\usepackage{amssymb}
 \usepackage{amsthm}

\usepackage[utf8]{inputenc}
\usepackage{graphicx}

\usepackage{nicematrix}
\usepackage{colortbl}
\definecolor{Gray}{gray}{0.9}
\usepackage{ragged2e}
\newcolumntype{g}{>{\columncolor{Gray}}r}
\newcolumntype{h}{>{\columncolor{Gray}}r}

\usepackage{arxiv}
\usepackage{relsize}
\usepackage[capitalize]{cleveref}
\usepackage{algorithm}
\usepackage{algpseudocode}
\usepackage{ifthen}
\usepackage[font=normalsize]{subcaption}
\usepackage{color}
\usepackage{wrapfig}
\usepackage{float}
\usepackage{bm}
\usepackage{multirow}
\usepackage{array} 
\graphicspath{{images/}}

\usepackage{mathtools}

\newtheorem{theorem}{Theorem}[section]

\newtheorem{proposition}[theorem]{Proposition}
\newtheorem{definition}[theorem]{Definition}
\newtheorem{remark}[theorem]{Remark}
\usepackage{amsopn}

\makeatletter
\newcommand{\srcsize}{\@setfontsize{\srcsize}{5pt}{5pt}}
\makeatother

\newcommand{\abs}[1]{\left\lvert#1\right\rvert}
\newcommand{\norm}[1]{\left\lVert#1\right\rVert}

\newcommand{\OO}[1]{\ifthenelse{\equal{#1}{0}}{\mathrm{O}(h)}{
		\ifthenelse{\equal{#1}{+1}}{\mathrm{O}(h)}{
			\ifthenelse{\equal{#1}{-1}}{\mathrm{O}h)}{
				\mathrm{O}(h)
}}}}  % INPUT i k
\title{A domain splitting strategy for solving PDEs}%\tnoteref{t1}}

\author{
Ken Trotti\\
Faculty of Informatics\\
University of Italian Switzerland\\
ken.trotti@usi.ch
}

\begin{document}
\maketitle
	
\begin{abstract}
In this work we develop a novel domain splitting strategy for the solution of partial differential equations. Focusing on a uniform discretization of the $d$-dimensional advection-diffusion equation, our proposal is a two-level algorithm that merges the solutions obtained from the discretization of the equation over highly anisotropic submeshes to compute an initial approximation of the fine solution. The algorithm then iteratively refines the initial guess by leveraging the structure of the residual. Performing costly calculations on anisotropic submeshes enable us to reduce the dimensionality of the problem by one, and the merging process, which involves the computation of solutions over disjoint domains, allows for parallel implementation.
\end{abstract}


%\end{frontmatter}

%\linenumbers

\section{Introduction}
Supercomputer nowadays are high-performance computers with thousands and even millions of cores, that are designed to handle extremely large and complex computing tasks. This makes parallelization, or the ability to divide a task into smaller parts that can be run simultaneously on different cores, an important factor in maximizing the performance of a supercomputer. To effectively exploit this feature, an algorithm must be designed accordingly. In the case of high-dimensional partial differential equations, there are several standard approaches. These include Deep Learning techniques \cite{deeplearning,deeplearning2,deeplearning3}, Monte Carlo methods \cite{montecarlo,montecarlo2}, Sparse grids \cite{sparsegrids}, and other methods consists in exploiting properties of the domain, symmetries, regularity, see e.g. \cite{symmetry}.\\
The Sparse Grids (SG) method starts with a uniform mesh with $N^d$ grid points and then removes certain points according to a specific rule, yielding a sparse mesh with $O\big(N\log^{d-1} N\big)$ grid points. While this approach allows for low memory requirements, it also leads to an increased approximation error compared to using an uniform mesh. Specifically, when a second-order accurate scheme is used, the resulting approximation error in the $L^2$ norm is $O\big(N^{-2}\log(N)^{d-1}\big)$, which is larger than the error of $O\big(N^{-2}\big)$ achieved with the uniform mesh.\\
Our goal is to develop an approach that utilizes a coarser mesh than a uniform one, but avoids the increase in error that is characteristic of SG. Therefore, we aim to develop an algorithm that allows for efficient and parallel computation while maintaining a high level of accuracy and potentially low memory requirements. To achieve this goal, we introduce the \textit{Anisotropic Submeshes Domain Splitting Method} (ASDSM), a novel two-level domain splitting method for the parallel solution of $d$-dimensional PDEs. 
ASDSM involves dividing a fine uniform mesh into multiple strongly anisotropic submeshes. The solutions obtained from each submesh are then merged together forming an approximate solution of the discretized PDE on the uniform mesh. The process is repeated by exploiting the structure of the residual to improve the accuracy of the solution.
This approach allows us to effectively reduce the dimensionality of the problem by a factor of one, allowing for the efficient solution of larger and more complex problems. Additionally, as discussed in Section \ref{subsec:Efficient_implementation_and_final_remarks} a smart memory handling strategy can further reduce the memory requirements compared to using the original uniform mesh.

In this preliminary research, we focus on the well-known advection-diffusion equation, which is a fundamental mathematical model used in a variety of fields including fluid dynamics, heat transfer, and contaminant transport \cite{fluiddynamics,heattransfer,contaminanttransport}. 
Consider the following $d$-dimensional advection-diffusion equation with Dirichlet boundary conditions (BCs)
\begin{equation}\label{eq:genericPDE}
	\begin{cases}
		\begin{tabular}{ll}
			$-\sum_{i=1}^d\alpha_i \frac{\partial u}{\partial x_i^2}(\mathbf{x})+\beta_i \frac{\partial u}{\partial x_i}(\mathbf{x})=s(\mathbf{x}),$ & $\mathbf{x}\in\Omega\subset\mathbb{R}^{d}$,\\
			$u(\mathbf{x})=g(\mathbf{x}),$& $\mathbf{x}\in\partial\Omega,$
		\end{tabular}
	\end{cases}
\end{equation}
and its time-dependent counterpart
\begin{equation}\label{eq:genericPDE_Time}
	\begin{cases}
		\begin{tabular}{ll}
			$\frac{\partial u}{\partial t}(\mathbf{x},t)-\sum_{i=1}^d\alpha_i \frac{\partial u}{\partial x_i^2}(\mathbf{x})+\beta_i \frac{\partial u}{\partial x_i}(\mathbf{x})=s(\mathbf{x},t),$ & $(\mathbf{x},t)\in\Omega\times[0,1]\subset\mathbb{R}^{d+1}$,\\
			$u(\mathbf{x})=g(\mathbf{x},t),$& $(\mathbf{x},t)\in\partial\Omega\times[0,1),$
		\end{tabular}
		%		\mathcal{F}(\mathbf{x})=s(\mathbf{x}),\quad\ \ \ \! \mathbf{x}\in\Omega,\\
		%		\frac{\partial u}{\partial t}(\mathbf{x},t)-\nabla u(x)=s(\mathbf{x},t),\quad\ \ \ \! (\mathbf{x},t)\in\Omega\times[0,1]\subset\mathbb{R}^{d+1},\\
		%		u(\mathbf{x})=g(\mathbf{x},t),\qquad (\mathbf{x},t)\in\partial\Omega\times[0,1),
	\end{cases}
\end{equation}
where $s$ is the source term, $g$ the BCs, $\alpha>0$ the diffusion coefficient, $\beta$ the velocity field and $\Omega$ the space domain.

In the following sections, we will provide a detailed description of ASDSM. In details, in Section \ref{sec:intro_ASDSM} we introduce the basic tools needed to explain ASDSM. We provide the detailed construction of the involved meshes and projectors, and briefly sum up the discretization process of the PDE. In Section \ref{sec:2D_ASDSM}, we provide a step-by-step explanation with figures to illustrate the ASDSM Algorithm. We also provide some remarks on how to implement the algorithm efficiently. In Section \ref{sec:Results}, we present numerical results demonstrating the effectiveness of our method in terms of residual reduction per iteration. %We compare our results to other methods and show the improvements in accuracy and computational efficiency achieved by using our approach.
Finally, in Section \ref{sec:Conclusions}, we draw conclusions and outline potential areas for future work. We discuss the implications of our method and how it can be applied to other problems in the field.

%In the following, in Section \ref{sec:intro_ASDSM} we introduce the notation and meshes and projectors needed to define the ASDSM algorithm, and then we discretize the PDE. In Section \ref{sec:2D_ASDSM} we describe ASDSM in details through a step-by-step guide with figures. Finally, in Section \ref{sec:Results} we provide numerical evidences to support this approach and in Section \ref{sec:Conclusions} we draw conclusions.

\section{Introduction to the ASDSM Algorithm}\label{sec:intro_ASDSM}
%For the sake of simplicity we restrict to the PDE in \eqref{eq:genericPDE} with $d=2$ and square domain $\Omega=[0,1]\times[0,1]$.
%
%Before giving the ASDSM Algorithm, in this section we first introduce the involved meshes and fix the notation, then we discretize PDE \eqref{eq:genericPDE} through the finite differences method, and finally we introduce the projectors which we will need to project the unknown onto different meshes.
For the sake of simplicity, in this section we will consider the two-dimensional case of the PDE in \eqref{eq:genericPDE} on the square domain $\Omega=[0,1]\times[0,1]$.

Before providing the ASDSM Algorithm, we first introduce the necessary meshes and fix the notation. We then discretize the PDE using the finite differences method, and introduce the projectors that will be used to project the unknown onto different meshes. These tools form the foundation upon which the ASDSM Algorithm is built, and will be essential for understanding the details of the method.

\subsection{Meshes}\label{subsec:meshes}
Let $N^{(x)}_c,N^{(x)}_f,N^{(y)}_c,N^{(y)}_f\in\mathbb{N}$ such that 
\begin{equation}\label{eq:condition_meshes}
	\begin{split}
		N^{(x)}_f+1=n_x(N^{(x)}_c+1),\\
		N^{(y)}_f+1=n_y(N^{(y)}_c+1),
	\end{split}
\end{equation} 
for some $n_x,n_y\in\mathbb{N}$. Then consider the step lengths
\begin{equation}
	\begin{tabular}{ll}
		$h_x=\frac{1}{N^{(x)}_f+1}$, & $h_y=\frac{1}{N^{(y)}_f+1}$,\\
		$H_x=\frac{1}{N^{(x)}_c+1}=n_xh_x$,\qquad\  & $H_y=\frac{1}{N^{(y)}_c+1}=n_yh_y$,
	\end{tabular}
\end{equation} 
and the following uniform meshes
\begin{equation}\label{eq:mesh2D}
	\begin{split}
		\Omega^{h_x,h_y}&=\lbrace (x_i,y_j) \,\big|\, x_i=ih_x,\ y_j=jh_y,\ i=1,...,N^{(x)}_f,\ j=1,...,N^{(y)}_f\rbrace;%\label{eq:mesh2D_fine}
		\\
		\Omega^{h_x,H_y}&=\lbrace (x_i,y_j) \,\big|\, x_i=ih_x,\ y_j=jH_y,\ i=1,...,N^{(x)}_f,\ j=1,...,N^{(y)}_c\rbrace;%\label{eq:mesh2D_dense_x}
		\\
		\Omega^{H_x,h_y}&=\lbrace (x_i,y_j) \,\big|\, x_i=iH_x,\ y_j=jh_y,\ i=1,...,N^{(x)}_c,\ j=1,...,N^{(y)}_f\rbrace;%\label{eq:mesh2D_dense_y}
		\\
		\Omega^{H_x,H_y}&=\lbrace (x_i,y_j) \,\big|\, x_i=iH_x,\ y_j=jH_y,\ i=1,...,N^{(x)}_c,\ j=1,...,N^{(y)}_c\rbrace;%\label{eq:mesh2D_coarse}
		\\
		\Omega^{h_x,h_y}_{i,j}&=
		\begin{Bmatrix}
			\begin{tabular}{l}
				%	\ \\
				%	$(x_k,y_l)$\\
				%	\ 
				\hspace{-0.1cm}\multirow{2}{*}{$(x_k,y_l)$}\\ \ 
			\end{tabular}
			\hspace{-0.4cm}& \Bigg|
			\begin{tabular}{lll}
				$x_k=iH_x+kh_x,$ & $k=1,...,n_x-1, $ & $i=0,...,N^{(x)}_c$\\
				$y_l=jH_y+lh_y,$ & $l=1,...,n_y-1, $ & $j=0,...,N^{(y)}_c$
			\end{tabular}\hspace{-0.1cm}
		\end{Bmatrix}.
		%\Omega^{h_x,h_y}_{i,j}&=\lbrace (x_k,y_l) \,\big|\, x_k=x_i+kh_x,\ y_l=y_j+lh_y,\ k=1,...,n_x-1,\ l=1,...,n_y-1\rbrace.%\label{eq:mesh2D_block}
	\end{split}
\end{equation}
We will refer to the meshes in equation \eqref{eq:mesh2D} respectively 
% \cref{eq:mesh2D_fine,eq:mesh2D_dense_x,eq:mesh2D_dense_y,eq:mesh2D_coarse},
as the fine mesh, the anisotropic meshes dense along the $x$ and $y$ axis, the coarse mesh and the ``holes''. Moreover, the union between $\Omega^{h_x,H_y}\cup\Omega^{H_x,h_y}$ is called \emph{skeleton} or \emph{coarse structure}. Furthermore, we denote by $\Omega^{h_x,h_y}_{holes}$ the union of all holes, i.e. $\Omega^{h_x,h_y}_{holes}:=\cup_{i,j=0}^{N^{(x)}_c,N^{(y)}_c}\Omega^{h_x,h_y}_{i,j}$.

The following proposition, which is a consequence of the condition in equation~\eqref{eq:condition_meshes}, shows the relations between the meshes in \eqref{eq:mesh2D}.
\begin{proposition}\label{prop:link_between_meshes}
	Consider the meshes in equation \eqref{eq:mesh2D}, %equations \cref{eq:mesh2D_fine,eq:mesh2D_dense_x,eq:mesh2D_dense_y,eq:mesh2D_coarse},
	then the following relations hold
	\begin{itemize}
		\item[a)] $\Omega^{h_x,H_y}\subset\Omega^{h_x,h_y}$, \qquad $\Omega^{H_x,h_y}\subset\Omega^{h_x,h_y}$;
		\item[b)] 
		$\Omega^{H_x,H_y}\subset\Omega^{h_x,h_y},\qquad\Omega^{H_x,H_y}\subset\Omega^{H_x,h_y},\qquad\Omega^{H_x,H_y}\subset\Omega^{h_x,H_y}$;
		\item[c)] $\Omega^{h_x,h_y}_{i,j}\cap \Omega^{h_x,h_y}_{k,l}=\emptyset$\quad if\quad $(i,j)\neq(k,l)$;
		\item[d)]
		%$\bigcup_{i,j=0}^{N^{(x)}_c,N^{(y)}_c}
		%\Omega^{h_x,h_y}_{i,j} 
		$\Omega^{h_x,h_y}_{holes}
		=\Omega^{h_x,h_y} \setminus \left( \Omega^{h_x,H_y} \cup \Omega^{H_x,h_y}\right)$.
		
	\end{itemize}
\end{proposition}
In Proposition \ref{prop:link_between_meshes}, item (a) show that the two anisotropic meshes are sub-meshes of the fine one, and item (b) shows that the coarse mesh is a sub-mesh of the fine mesh and the anisotropic meshes. Moreover, from item (a) it follows that the mesh $\Omega^{h_x,H_y} \cup \Omega^{H_x,h_y}$ is a \emph{non-uniform} sub-mesh of the fine one from which, by adding $\Omega^{h_x,h_y}_{holes}$, we obtain the fine mesh.\\
Indeed, equivalently to item (d), we have
$$\Omega^{h_x,h_y}=\left(\bigcup_{i,j=0}^{N^{(x)}_c,N^{(y)}_c}
\Omega^{h_x,h_y}_{i,j}\right)\cup\Omega^{h_x,H_y} \cup \Omega^{H_x,h_y}.
$$
%This is the reason why we refer to meshes $\Omega^{h_x,h_y}_{i,j}$ as the ``holes''.

Let $N_c^{(x)}=N_c^{(y)}=2,\ N_f^{(x)}=N_f^{(y)}=8$. Then Figure \ref{fig:meshes} shows a graphical representation of the meshes in equation \eqref{eq:mesh2D}. 

\begin{figure}
	\centering
	\vspace{3pt}
	\includegraphics[width=0.8\linewidth]{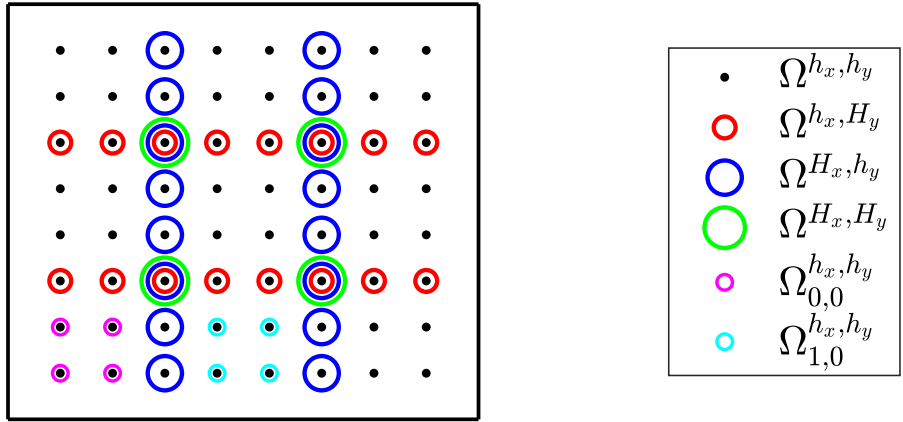}
	\caption{Graphical representation of the meshes.}\label{fig:meshes}
	\vspace{-0.4cm}
\end{figure}

\subsection{Projectors}
The ASDSM Algorithm aims to compute the discrete fine solution by solving equations over coarser meshes. In order to do this, we need to introduce the projectors that will allow us to transfer a discrete function from one mesh to another. These projectors will be essential for accurately transferring information between the different meshes used in our method.
%In this section, we define the projectors that will be used to transfer elements from one set to another.

We begin by introducing the set of all discrete functions defined over the fine mesh $\Omega^{h_x,h_y}$, in equation \eqref{eq:mesh2D}, as $V^{h_x,h_y}=\lbrace v\ |\ v\in\mathbb{R}^{\#\Omega^{h_x,h_y}}\rbrace$. Consider $v\in V^{h_x,h_y}$ to be an ordered vector $$v=[v_{1,1},v_{2,1},...,v_{N_f^{(x)},1},v_{1,2},...,v_{N_f^{(x)},2},...,v_{N_f^{(x)},N_f^{(y)}}]^\mathrm{T},$$
with $v_{i,j}$ being the value of the discrete function $v$ at the grid point $(x_i,y_j)\in \Omega^{h_x,h_y}$. We also define the sets $V^{h_x,H_y},V^{H_x,h_y},V^{H_x,H_y},V^{h_x,h_y}_{i,j}$ with similarly ordered elements. 
\begin{remark}\label{rem:Ordering}
	Clearly, for every position $m$ of $v\in V^{h_x,h_y}$ there exist unique indexes $i\in\lbrace 1,...,N_f^{(x)}\rbrace$ and $j\in\lbrace 1,...,N_f^{(y)}\rbrace$ such that $v_m=v_{i,j}$ with $(i-1)N_f^{y}+j=m$.%$v_{i(m),j(m)}$.
	%Without loss of generality, consider $v\in V^{h_x,h_y}$ with the aforementioned ordering. Then, for every position $m$ of vector $v$ there $\exists$ unique indexes $i\in\lbrace 1,...,N_f^{(x)}\rbrace$ and $j\in\lbrace 1,...,N_f^{(y)}\rbrace$ such that the position $m$ of vector $v$ is given by $v_{i(m),j(m)}$.
\end{remark}

\begin{definition}\label{def:projectors}
	Let us define the following projectors.
	\begin{itemize}
		\item[1)] $P_{ff}^{fc}: V^{h_x,h_y}\rightarrow V^{h_x,H_y}$, as the projector that maps vectors $v\in V^{h_x,h_y}$ into vector $w\in V^{h_x,H_y},\ w=P_{ff}^{fc}v$, by only keeping the components $v_{i,j}$ with $(x_i,y_j)\in\Omega^{h_x,H_y}$;
		Formally, from Remark \ref{rem:Ordering}, 
		\begin{equation*}
			\left(P_{ff}^{fc}\right)_{m,n}=
			\begin{cases}
				1\quad \text{if}\ \Omega^{h_x,h_y}\ni(x_{i(m)},y_{j(m)})=(x_{i(n)},y_{j(n)})\in\Omega^{h_x,H_y},\\
				0\quad \text{otherwise.}
			\end{cases}
		\end{equation*}
		%	\item[1)] $P_{ff}^{fc}: V^{h_x,h_y}\rightarrow V^{h_x,H_y}$, as the projector that maps vectors $v\in V^{h_x,h_y}$ into vector $w\in V^{h_x,H_y},\ w=P_{ff}^{fc}v$, by only keeping the components $v_{i,j}$ with $(x_i,y_j)\in\Omega^{h_x,H_y}$;
		\item[2)] $P_{ff}^{cf}: V^{h_x,h_y}\rightarrow V^{H_x,h_y}$, as the projector that maps vectors $v\in V^{h_x,h_y}$ into vector $w\in V^{H_x,h_y},\ w=P_{ff}^{cf}v$, by only keeping the components $v_{i,j}$ with $(x_i,y_j)\in\Omega^{H_x,h_y}$;
		\item[3)] $P_{ff}^{cc}: V^{h_x,h_y}\rightarrow V^{H_x,H_y}$, as the projector that maps vectors $v\in V^{h_x,h_y}$ into vector $w\in V^{H_x,H_y},\ w=P_{ff}^{cc}v$, by only keeping the components $v_{i,j}$ with $(x_i,y_j)\in\Omega^{H_x,H_y}$;
		\item[4)] $P_{ff}^{holes}: V^{h_x,h_y}\rightarrow V^{h_x,h_y}_{holes}$, as the projector that maps vectors $v\in V^{h_x,h_y}$ into vector $w\in V^{h_x,h_y}_{holes},\ w=P_{ff}^{holes}v$, by only keeping the components $v_{i,j}$ with $(x_i,y_j)\in\Omega^{h_x,h_y}_{holes}$;
	\end{itemize}
\end{definition}
Since the transpose of a projector switches domain with codomain, we note that:
\begin{itemize}
	\item[1)] $\left(P_{ff}^{fc}\right)^\mathrm{T}=P_{fc}^{ff}$;
	\item[2)] $\left(P_{ff}^{cf}\right)^\mathrm{T}=P_{cf}^{ff}$;
	\item[3)] $\left(P_{ff}^{cc}\right)^\mathrm{T}=P_{cc}^{ff}$;
	\item[4)] $\left(P_{ff}^{holes}\right)^\mathrm{T}=P_{holes}^{ff}$.
\end{itemize}
This allows us to retrieve the properties listed in Proposition \ref{prop:properties_proj}.
\begin{proposition}\label{prop:properties_proj}
	Let $I$ be the identity matrix of the opportune size. Then the projectors defined above have the following properties:
	\begin{itemize}
		\item[1)] $P_{ff}^{fc}\left(P_{ff}^{fc}\right)^\mathrm{T}=P_{ff}^{fc}P_{fc}^{ff}=I$;
		\item[2)] $P_{ff}^{cf}\left(P_{ff}^{cf}\right)^\mathrm{T}=P_{ff}^{cf}P_{cf}^{ff}=I$;
		\item[3)] $P_{ff}^{cc}\left(P_{ff}^{cc}\right)^\mathrm{T}=P_{ff}^{cc}P_{cc}^{ff}=I$;
		\item[4)] $P_{ff}^{holes}\left(P_{ff}^{holes}\right)^\mathrm{T}=P_{ff}^{holes}P_{holes}^{ff}=I$;
		\item[5)] Every projector has norm-1 and infinity norm equal to 1.
	\end{itemize}
\end{proposition}

\subsection{Discretization}
The discretization of the PDE over the uniform meshes in \eqref{eq:mesh2D} is done through the second order accurate centered finite differences method, and yields the following linear systems.
\begin{equation}\label{eq:linear_systems}
	\begin{tabular}{ll}
		$A_{ff}u_{ff}=b_{ff},$ & $u_{ff}\in\mathbb{R}^{N^{(x)}_fN^{(x)}_f};$ \\
		$A_{fc}u_{fc}=b_{fc},$ & $u_{fc}\in\mathbb{R}^{N^{(x)}_fN^{(x)}_c};$ \\
		$A_{cf}u_{cf}=b_{cf},$ & $u_{cf}\in\mathbb{R}^{N^{(x)}_cN^{(x)}_f};$ \\
		$A_{cc}u_{cc}=b_{cc},$ & $u_{cc}\in\mathbb{R}^{N^{(x)}_cN^{(x)}_c};$ \\
		$A_{(ij)}u_{(ij)}=b_{(ij)},$ & $u_{(ij)}\in\mathbb{R}^{(n_x-1)(n_y-1)},\ i=1,...,n_x-1,\ j=1,...,n_y-1.$
	\end{tabular}
\end{equation}
In details,
\begin{equation*}
	\begin{split}
		A_{ff}=&\ \alpha\left(I_{N_f^{(y)}}\otimes \left(\frac{1}{h_x^2}L_{N_f^{(x)}}\right)+
		\left(\frac{1}{h_y^2}L_{N_f^{(y)}}\right) \otimes I_{N_f^{(x)}}\right) +\\
		&+\beta\left(I_{N_f^{(y)}}\otimes \left(\frac{1}{2h_x}D_{N_f^{(x)}}\right)+
		\left(\frac{1}{2h_y}D_{N_f^{(y)}}\right) \otimes I_{N_f^{(x)}}\right),\\
		b_{ff}=&[s_{1,1},s_{2,1},...,s_{N_f^{(x)},1},s_{1,2},...,s_{N_f^{(x)},N_f^{(y)}}]^{\mathrm{T}},
	\end{split}
\end{equation*}
where $\otimes$ is the tensor product, $s_{i,j}=s(x_i,y_j),\ (x_i,y_j)\in\Omega^{h_x,h_y}$ and 
$$D_N=\begin{bmatrix}
	0 & 1 & &&\\
	-1& 0 & 1 &  &\\
	& \ddots & \ddots & \ddots & \\
	&&-1&0&1\\
	& & & -1 & 0
\end{bmatrix}_{N\times N},\qquad
L_N=\begin{bmatrix}
	2 & -1 & &&\\
	-1& 2 & -1 &  &\\
	& \ddots & \ddots & \ddots & \\
	&&-1&2&-1\\
	& & & -1 & 2
\end{bmatrix}_{N\times N}.$$
The other matrices and known vectors in equation \eqref{eq:linear_systems} are defined similarly.

\begin{remark}\label{rem:sub_blocks_of_main_diagonal}
	Since $\Omega^{h_x,h_y}_{i,j}\subset \Omega^{h_x,h_y}$, then the coefficient matrix $A_{(ij)}$ is a main diagonal sub-block of matrix $A_{ff}$ if the same discretization scheme is applied.
\end{remark}

\section{The 2D ASDSM Algorithm}\label{sec:2D_ASDSM}
%In this section we give the ASDSM Algorithm and a detailed step-by-step explanation with figures. In the following figures we consider PDE \eqref{eq:genericPDE} with $\alpha=1$, $\beta=0$, and $g(x,y)$, $s(x,y)$ computed from the exact solution $u(x,y)=\sin(x\pi)\sin(2y\pi)$. Because of the didactic purpose, the equation is discretized over a very coarse mesh with $N_f^{(x)}=N_f^{(y)}=24$ and $N_c^{(x)}=N_c^{(y)}=4$.
%
%Note, the inputs of the following algorithms are only the variable inputs, which change with the iteration. For the sake of readability, other inputs that do not depend on the iteration are not shown. Moreover, substeps $k.1)$ and $k.2)$ of step $k$ means that step $k$ is divided into 2 disjoint parts that can be run in parallel.
In this section, we present the ASDSM Algorithm and provide a detailed step-by-step explanation with figures. We consider the two-dimensional advection-diffusion equation, given in equation \eqref{eq:genericPDE}, with $\alpha=1$, $\beta=0$, i.e., the Poisson problem, and $g(x,y)$ and $s(x,y)$ computed from the exact solution $u(x,y)=\sin(x\pi)\sin(2y\pi)$. For the sake of clarity, we discretize the equation using a very coarse mesh with $N_f^{(x)}=N_f^{(y)}=24$ and $N_c^{(x)}=N_c^{(y)}=4$.

It is important to note that the inputs of the following algorithms are only the variables that change with each iteration. To improve readability, we do not include other inputs that are constant while iterating through the algorithm. Additionally, sometimes the step $k$ can be divided into two disjoint parts that can be executed concurrently, namely substeps $k.1)$ and $k.2)$.

\subsection{Computation of the initial guess}\label{subsec:initial_guess}
The core of our proposal is a peculiar way to compute an initial guess. This approach, outlined in Algorithm \ref{algorithm:InitialGuess}, consists in computing the solutions $u_{fc},u_{cf},u_{cc}$ of the anisotropic and coarse linear systems in equation \eqref{eq:linear_systems}, merging together the three solutions, thus forming the approximated \textit{coarse structure} or \textit{skeleton} of the fine solution $u_{ff}$, and finally filling in the empty regions.

\begin{algorithm}%[H]
	\caption{Initial Guess Builder}\label{algorithm:InitialGuess}
	\begin{algorithmic}
		\Function{InitialGuess}{$b_{ff},b_{fc},b_{cf},b_{cc}$}
		%		\State $\tilde{u}_{ff}=\textbf{InitialGuess}\left(b_{ff},b_{fc},b_{cf},b_{cc}\right)$
		\State 1.1) $u_{fc}=(A_{fc})^{-1}b_{fc}$ \Comment{Solve the equation on $\Omega^{h_x,H_y}$}
		\State 1.2) $u_{cf}=(A_{cf})^{-1}b_{cf}$ \Comment{Solve the equation on $\Omega^{H_x,h_y}$}
		\If{*$b_{cc}$ is in input*} 
		\State 1.3) $u_{cc}=(A_{cc})^{-1}b_{cc}$ \Comment{Solve the equation on $\Omega^{H_x,h_y}$}
		\Else
		\State 2.1) $u_{fc}=\frac{u_{fc}}{\norm{u_{fc}}}$ \Comment{Normalize the vector}
		\State 2.2) $u_{cf}=\frac{u_{cf}}{\norm{u_{cf}}}$
		\EndIf
		\State 3) $\tilde{u}_{ff}=\text{Skeleton}\left(u_{fc},u_{cf},u_{cc}\right)$ \Comment{Builds the skeleton of the initial guess}
		\State 4) $\tilde{u}_{ff}=\text{Filler}\left(\tilde{u}_{ff},b_{ff}\right)$ \Comment{Fills the holes $\Omega^{h_x,h_y}_{i,j}$}
		\State\Return $\tilde{u}_{ff}$
		\EndFunction
	\end{algorithmic}
\end{algorithm}

\noindent In details, in step 1), the solutions $u_{fc},u_{cf},u_{cc}$ can be efficiently computed through ad-hoc solvers. For example, circulant preconditioners, incomplete LU preconditioners, and multigrid methods have already been shown to be valid preconditioners/solvers for structured anisotropic \cite{circanisotropic,iluanisotropic,mgmanisotropic} and isotropic \cite{circisotropic,iluisotropic,mgmisotropic} linear systems. \\
Note that, when computing the initial guess, the known term $b_{cc}$ is in input, hence the "if" condition is satisfied. The case where $u_{cc}$ is not provided, leading to skip step 1.3) and normalize the two solution in step 2), will be treated later.\\
% Later on, when dealing with the initial guess of the error equation, Step 1.3) will be skipped and the normalization in step 2) will be required.
%circulant preconditioners \cite{} and multigrid methods like V-Cycle with semi-coarsening \cite{} and the more robust \textit{multigrid as smoother} MG-S \cite{}, has been shown to be valid solvers in our case of anisotropic structured linear systems. While, for the isotropic one, a standard V-Cycle would be enough
After the computation of $u_{fc},u_{cf},u_{cc}$, respectively shown in red, blue, and green in Figure \ref{fig:solutions}, in step 3) we merge them into a unique unknown $\tilde{u}_{ff}$ through the Skeleton Builder Algorithm \ref{algorithm:Skeleton}. 

\begin{figure}[t]
	\centering
	\begin{subfigure}{0.48\textwidth}
		\includegraphics[width=1\linewidth]{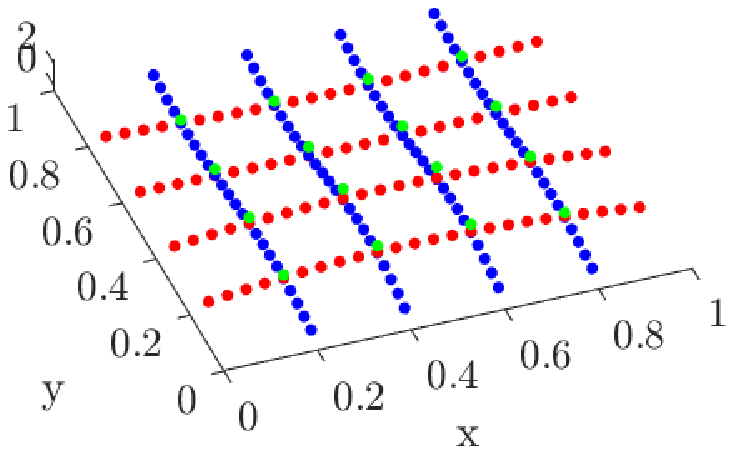}
		\caption{$u_{fc},u_{cf},u_{cc}$ in red, blue and green, respectively, over the meshes in Figure \ref{fig:meshes}.}
		\label{fig:solutions}
	\end{subfigure}\hfill
	\begin{subfigure}{0.48\textwidth}
		\includegraphics[width=1\linewidth]{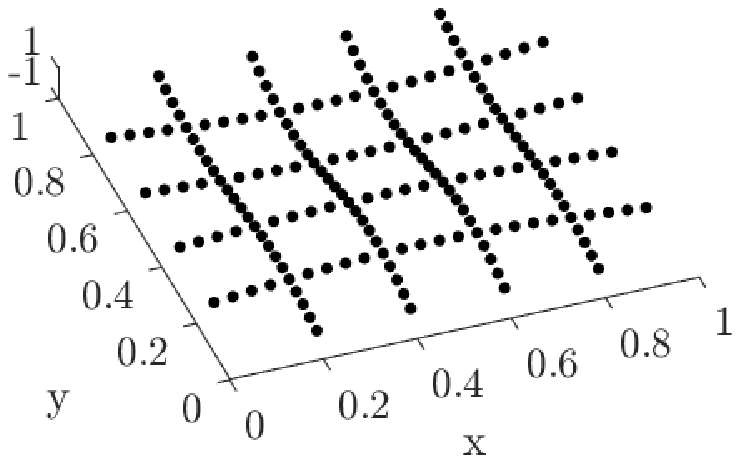}
		\caption{Skeleton of the surface over the mesh $\Omega^{h_x,H_y}\cup\Omega^{H_x,h_y}$ after the correction.}
		\label{fig:skeleton}
	\end{subfigure}
	\caption{Coarse structure before (a) and after (b) the merging process.}
\end{figure}

\begin{algorithm}%[H]
	\caption{Skeleton Builder}\label{algorithm:Skeleton}
	\begin{algorithmic}
		\Function{Skeleton}{$u_{fc},u_{cf},u_{cc}$}
		%		\State $\tilde{u}_{ff}=\textbf{}\left(\right)$
		\State 1.1) $\tilde{u}_{cc}^{(1)}=P_{fc}^{cc}u_{fc}$ 
		\State 1.2) $\tilde{u}_{cc}^{(2)}=P_{cf}^{cc}u_{cf}$ \Comment{Project the anisotropic solutions over the coarse mesh}
		\If{*$u_{cc}$ is in input*} 	\Comment{Richardson Extrapolation (RE) can be used}
		\State 2) $\hat{u}_{cc}=$RE$\left(\tilde{u}_{cc}^{(1)},\tilde{u}_{cc}^{(2)},u_{cc}\right)$ \Comment{Extrapolated coarse solution}
		\Else
		\State 3) $\hat{u}_{cc}=\tilde{u}_{cc}^{(1)}\quad | \quad \hat{u}_{cc}=\tilde{u}_{cc}^{(2)}$ \Comment{Choose the coarse solution}
		\EndIf
		\State 4.1) $e_{cc}^{(1)}=\hat{u}_{cc}-\tilde{u}_{cc}^{(1)}$
		\State 4.2) $e_{cc}^{(2)}=\hat{u}_{cc}-\tilde{u}_{cc}^{(2)}$ \Comment{Compute the error with respect to $\hat{u}_{cc}$}
		\State 5.1) $\tilde{u}_{fc}=u_{fc}+$Spline$(e_{cc}^{(1)},\Omega^{h_x,H_y})$
		\State 5.2) $\tilde{u}_{cf}=u_{cf}+$Spline$(e_{cc}^{(2)},\Omega^{H_x,h_y})$ \Comment{Correct the two solutions}
		\State 6) $\tilde{u}_{ff}=P_{fc}^{ff}\tilde{u}_{fc}+P_{cf}^{ff}\tilde{u}_{cf}-P_{cc}^{ff}P_{cf}^{cc}\tilde{u}_{cf}$
		\State\Return $\tilde{u}_{ff}$
		\EndFunction
	\end{algorithmic}
\end{algorithm}
In Figure \ref{fig:solutions} we note that, due to the different meshes involved, the three solution do not perfectly overlap at the cross-points. Therefore, Algorithm \ref{algorithm:Skeleton} exploits Richardson extrapolation to compute accurate cross-points, then adjusts $u_{fc}$ and $u_{cf}$, and merges them together.\\
%forming the so-called skeleton of the surface.\\
In details, in step 1) we project the two anisotropic solutions $u_{fc},u_{cf}$ over the coarse isotropic mesh $\Omega^{H_x,H_y}$ through the projectors $P_{fc}^{cc},P_{cf}^{cc}$. Then, since $u_{cc}$ is in input, in step 3) we use Richardson extrapolation \cite{richardsonextrapolation} to generate a more accurate solution $\hat{u}_{cc}$ over $\Omega^{H_x,H_y}$. 
\begin{remark}
	In order to further increase the accuracy of $\hat{u}_{cc}$, if the smoothness requirements are met, one can compute more solutions, cheaper than $u_{fc},u_{cf}$, e.g,  over the anisotropic meshes $\Omega^{kh_x,H_y}$ or $\Omega^{H_x,kh_y}$, for some  $k\in\mathbb{N}$.
	%	with $k\in\mathbb{N}$ such that there exists a $l\in\mathbb{N}$ with $lkh_x=H_x$ or $lkh_y=H_y$, hence $\Omega^{H_x,H_y}\subset\Omega^{kh_x,H_y}$ or $\Omega^{H_x,H_y}\subset\Omega^{H_x,kh_y}$. 
	The more information, the higher the accuracy of the extrapolated solution.
\end{remark}
%The case where $u_{cc}$ is not provided as an input of Algorithm \ref{algorithm:Skeleton}, leading to step 3), will be treated later.
%We will deal with the case where $u_{cc}$ is not in input later.
%DEFINIRE QUESTO PROCESSO CON PIU DETTAGLI DA QUALCHE PARTE NELLA SEZIONE PRECEDENTE

Now that the accurate cross-points have been computed, we have to adjust the two anisotropic solutions $u_{fc},u_{cf}$ into $\tilde{u}_{fc},\tilde{u}_{cf}$ such that $P_{fc}^{cc}\tilde{u}_{fc}=P_{cf}^{cc}\tilde{u}_{cf}=\hat{u}_{cc}$, i.e., their projection onto the coarse isotropic mesh coincides with the extrapolated cross-points $\hat{u}_{cc}$. To do so, in step 4) we compute the distance between $P_{fc}^{cc}u_{fc},P_{cf}^{cc}u_{cf}$ and $\hat{u}_{cc}$. Then, in step 5), we interpolate the distance with zero BCs to meshes $\Omega^{h_x,H_y},\Omega^{H_x,h_y}$ and add it to $u_{fc},u_{cf}$, respectively. By using high order spline interpolation this process yields smooth unknowns $\tilde{u}_{fc},\tilde{u}_{cf}$ with the aforementioned property. At this point in step 6) we project the corrected anisotropic unknowns to the fine mesh $\Omega^{h_x,h_y}$, we sum them together and finally we remove the grid points which are counted twice. The newly defined unknown $\tilde{u}_{ff}$ is called the \textit{skeleton} or \textit{coarse structure} of our initial guess. \\
As shown in Figure \ref{fig:skeleton}, the skeleton surface is zero over the so-called \textit{empty} regions or \textit{holes} $\Omega_{i,j}^{h_x,h_y}$ (in white) and it is smooth over the coarse structure $\Omega^{h_x,H_y}\cup\Omega^{H_x,h_y}$ (in black). This property allows for the efficient filling of the holes, as outlined in the sketch of the Filler Algorithm in \ref{algorithm:Filler}.

\begin{algorithm}%[H]
	\caption{Filler}\label{algorithm:Filler}
	\begin{algorithmic}
		\Function{Filler}{$\tilde{u}_{ff},A_{ff},b_{ff}$}
		%		\State $\tilde{u}_{ff}=\text{Filler}\left(\tilde{u}_{ff},b_{ff}\right)$
		\State 1) $c=-P_{ff}^{holes}A_{ff}\tilde{u}_{ff}$ \Comment{Use $\tilde{u}_{ff}$ as BCs}
		\State 2) $v=\left(P_{ff}^{holes}A_{ff}P_{holes}^{ff}\right)^{-1}\left(P_{ff}^{holes}b_{ff}+c\right)$
		\State 3) $\tilde{u}_{ff}=\tilde{u}_{ff}+P_{holes}^{ff}v$
		\State\Return $\tilde{u}_{ff}$
		\EndFunction
	\end{algorithmic}
\end{algorithm}

The filling process consists in using the skeleton $\tilde{u}_{ff}$ as BCs for new disjoint discrete PDEs, solve them in parallel and finally update the fine unknowns.\\
In details, due to the zero filling of the holes, the product $A_{ff}\tilde{u}_{ff}$ in step 1) projects the information on the boundaries (the skeleton) inside the holes. Then, projector $P_{ff}^{holes}$ removes the skeleton from $c$ by only keeping the empty regions, which now include the information regarding the boundaries. Hence, vector $c$ contains the BCs for the linear system defined over the holes.\\
In step 2) we project matrix $A_{ff}$ and the known term $b_{ff}$ over the fine mesh without the skeleton. The solution of the newly obtained linear system is the chaining of the solutions inside each one of the holes.\\
Note that matrix $P_{ff}^{holes}A_{ff}P_{holes}^{ff}$ has a block structure. According to Remark \ref{rem:sub_blocks_of_main_diagonal}, since the same discretization scheme is used, each block of $P_{ff}^{holes}A_{ff}P_{holes}^{ff}$ corresponds to $A_{(ij)}$ for some $i,j$. Therefore, the inversion of the projected matrix can be easily parallelized.\\
Finally, in step 3) we update the skeleton with the new information. As shown in Figure \ref{fig:filled_skeleton}, this process yields a continuous unknown (unless the time variable is involved).

\begin{figure}
	\centering
	\begin{subfigure}{0.48\textwidth}
		\includegraphics[width=1\linewidth]{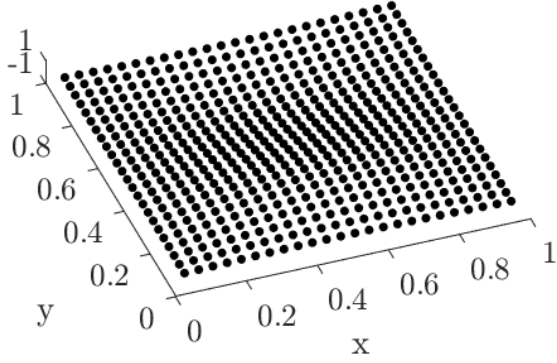}
		\caption{Initial guess obtained after the filling process of the skeleton.}
		\label{fig:filled_skeleton}
	\end{subfigure}\hfill
	\begin{subfigure}{0.48\textwidth}
		\includegraphics[width=1\linewidth]{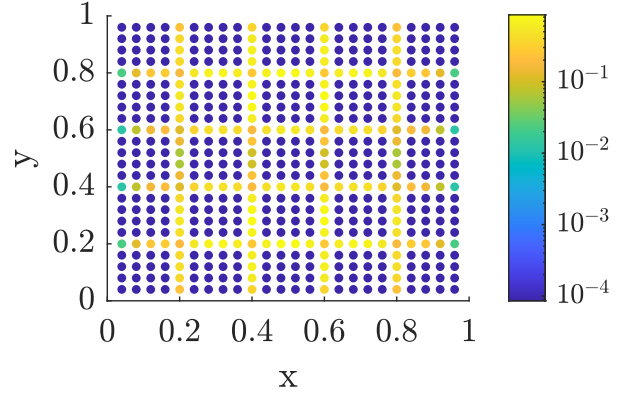}
		\caption{Reshaped residual of the initial guess.}
		\label{fig:residual}
	\end{subfigure}
	\caption{Initial guess (a) and respective residual (b).}
\end{figure}

%\subsection{Properties of the initial guess}
Before giving the ASDSM algorithm, we summarize three properties of the Initial Guess Algorithm \ref{algorithm:InitialGuess}.
\begin{proposition}\label{prop:zero_res_in_holes}
	Let $u_{ff}=\left(A_{ff}\right)^{-1}b_{ff}$ be the numerical solution to the fine linear system in equation \eqref{eq:linear_systems} and let $u^{(0)}$ be the initial guess computed through Algorithm \ref{algorithm:InitialGuess}, with $r$ being the residual vector reshaped into a discrete surface. Then
	\begin{itemize}
		\item[1)] the unknown $u^{(0)}$ coincides with $u_{ff}$ if and only if the skeleton of $u^{(0)}$ coincides with the skeleton of $u_{ff}$;%, formally
		%		\begin{equation*}
			%			SK(u^{(k)})=SK(u_{ff});
			%		\end{equation*}
		\item[2)] $r$ is zero over the holes and non-zero over its skeleton, unless $u^{(0)}=u_{ff}$;
		\item[3)] The skeleton of $r$ may be non-smooth near the coarse mesh points $\Omega^{H_x,H_y}$, i.e., the cross-points of the two anisotropic meshes.
	\end{itemize} 
	%	Moreover, if $SK(u^{(k)})\neq SK(u_{ff})$, then the residual may be non-smooth at the coarse mesh points $\Omega^{H_x,H_y}$, i.e., the cross-points of the two anisotropic meshes.
\end{proposition}
\begin{proof}
	\textbf{Statement 1)} If $u^{(0)}$ coincides with $u_{ff}$, clearly the two skeletons coincide. Conversely, if the two skeletons coincide then the output of Algorithm \ref{algorithm:Skeleton} is
	\begin{equation}\label{eq_temp1}
		%		\left(I-P_{holes}^{ff}P_{ff}^{holes}\right)u^{(0)}=\left(I-P_{holes}^{ff}P_{ff}^{holes}\right)u_{ff}
		\tilde{u}_{ff}=\left(I-P_{holes}^{ff}P_{ff}^{holes}\right)u_{ff},
	\end{equation}
	i.e., $\tilde{u}_{ff}=0$ over the holes $\Omega^{h_x,h_y}_{holes}$ and $\tilde{u}_{ff}=u_{ff}$ over the skeleton $\Omega^{h_x,H_y}\cup\Omega^{H_x,h_y}$. Then $\tilde{u}_{ff}$ passes through Algorithm \ref{algorithm:Filler} yielding
	\begin{equation*}
		\begin{split}
			%		u^{(0)}&=\tilde{u}_{ff}+P_{holes}^{ff}v\\
			%		&=\tilde{u}_{ff}+P_{holes}^{ff}\left(P_{ff}^{holes}A_{ff}P_{holes}^{ff}\right)^{-1}\left(P_{ff}^{holes}b_{ff}+c\right)\\
			%		&=\tilde{u}_{ff}+P_{holes}^{ff}\left(P_{ff}^{holes}A_{ff}P_{holes}^{ff}\right)^{-1}\left(P_{ff}^{holes}b_{ff}-P_{ff}^{holes}A_{ff}\tilde{u}_{ff}\right).
			u^{(0)}&=\tilde{u}_{ff}+P_{holes}^{ff}\left(P_{ff}^{holes}A_{ff}P_{holes}^{ff}\right)^{-1}\left(P_{ff}^{holes}b_{ff}-P_{ff}^{holes}A_{ff}\tilde{u}_{ff}\right),
		\end{split}
	\end{equation*}
	and from equation \eqref{eq_temp1}
	\begin{equation*}
		\begin{split}
			u^{(0)}&=\tilde{u}_{ff}+P_{holes}^{ff}\left(P_{ff}^{holes}A_{ff}P_{holes}^{ff}\right)^{-1}\left(P_{ff}^{holes}A_{ff}P_{holes}^{ff}P_{ff}^{holes}u_{ff}\right)\\
			&=\tilde{u}_{ff}+P_{holes}^{ff}P_{ff}^{holes}u_{ff},
		\end{split}
	\end{equation*}
	which is equal to $u_{ff}$ and proves statement 1).
	%	 Conversely, if the two skeletons coincide, the filling process fills the region $\Omega^{h_x,h_y}_{holes}$ using the skeleton as BC. Since $\Omega^{h_x,h_y}_{holes}$ consists in disjoint subdomains, and, as a consequence to Remark \ref{rem:sub_blocks_of_main_diagonal}, the discretization of the equation over $\Omega^{h_x,h_y}_{holes}$ yields a block matrix whose blocks are the main diagonal blocks of the main matrix $A_{ff}$ Therefore, the filled points are consistent with $u_{ff}$ because the filling process is consistent with the restriction of the PDE \eqref{eq:genericPDE} to each hole $\Omega^{h_x,h_y}_{i,j}$.\\
	
	\noindent	
	\textbf{Statement 2)} This is a consequence of Remark \ref{rem:sub_blocks_of_main_diagonal}. Indeed, consider $P_{ff}^{holes}r$, i.e., the projection of the residual over $\Omega_{holes}^{h_x,h_y}$. % the subdomains $\Omega_{i,j}^{h_x,h_y}$, $\forall i,j$. 
	Then, we have
	\begin{equation*}
		\begin{split}
			P_{ff}^{holes}r&=P_{ff}^{holes}\left(b_{ff}-A_{ff}u^{(0)}\right)\\
			&=P_{ff}^{holes}\left(b_{ff}-A_{ff}\left(  \tilde{u}_{ff}+P_{holes}^{ff}v  \right)\right)\\	
			&=P_{ff}^{holes}\left(b_{ff}-A_{ff}\left( 		\tilde{u}_{ff}+P_{holes}^{ff}\left(P_{ff}^{holes}A_{ff}P_{holes}^{ff}\right)^{-1}\left(P_{ff}^{holes}b_{ff}+c\right) \right)\right)\\
			&=P_{ff}^{holes}b_{ff}-P_{ff}^{holes}A_{ff}\tilde{u}_{ff}-P_{ff}^{holes}A_{ff}P_{holes}^{ff}\left(P_{ff}^{holes}A_{ff}P_{holes}^{ff}\right)^{-1}\left(P_{ff}^{holes}b_{ff}+c\right)\\
			&=P_{ff}^{holes}b_{ff}-P_{ff}^{holes}A_{ff}\tilde{u}_{ff}-P_{ff}^{holes}b_{ff}-c
		\end{split}
	\end{equation*}
	which is zero since $c=-P_{ff}^{holes}A_{ff}\tilde{u}_{ff}$. Hence, the residual over the holes is zero. Moreover, in the case where $u^{(0)}\neq u_{ff}$, the residual is non-zero and by the previous result it must be non-zero only over its skeleton.\\	
	\textbf{Statement 3)} It follows by the fact that the coarse points $(x_i,y_j)\in\Omega^{h_x,h_y}\cap\Omega^{H_x,H_y}$ are the cross points between 4 subdomains $\Omega^{h_x,h_y}_{i,j}$ for some $i,j$, where 4 solution are computed through the wrong BCs, since $u^{(0)}\neq u_{ff}$. %, if the skeleton does not coincide with the one of $u_{ff}$.
	Therefore, if $(x_i,y_j)$ is a cross-point between the two anisotropic meshes, there may be a significant difference between $r_{i,j}$ (calculated from elements of the skeleton only) and $r_{i\pm 1,j}$ or $r_{i,j\pm 1}$ (calculated from elements of two domains).
\end{proof}

\begin{remark}\label{rem:different_discretization}
	In the proof of Proposition \ref{prop:zero_res_in_holes}, there is no explicit connection established between the structure of the residual and the discretization method used for the anisotropic linear systems. As a result, using an alternative discretization scheme would still yield a residual with the same sparse structure.
	%	In the proof of Proposition \ref{prop:zero_res_in_holes} we did not mention any link between the structure of the residual and the discretization method of the anisotropic linear systems. Therefore, the use of a different discretization scheme would lead to a residual with the same sparse structure.
\end{remark}

\subsection{Iterative approach}\label{subsec:iterative_approacb}
The ASDSM Algorithm, summarized in Algorithm~\ref{algorithm:ASDSM}, consists in iterating the Initial Guess Algorithm~\ref{algorithm:InitialGuess} by exploiting the structure of the residual.\\

\begin{algorithm}%[H]
	\caption{ASDSM Algorithm}\label{algorithm:ASDSM}
	\begin{algorithmic}
		\Function{ASDSM}{}
		%		\State $\tilde{u}_{ff}=\textbf{ASDSM}$
		\State 1) $u^{(0)}=\textbf{InitialGuess}\left(b_{ff},b_{fc},b_{cf},b_{cc}\right)$ \Comment{Compute the initial guess}
		\State 2) $k=0$
		\While{$\|b_{ff}-A_{ff}u^{(k)}\|<$tol $\quad|\quad$ *residual stagnates*} 
		\State 3)\ \ $r_{ff}=b_{ff}-A_{ff}u^{(k)}$ \Comment{Compute the residual on $\Omega^{h_x,h_y}$}
		\State 4.1) $r_{fc}=P_{ff}^{fc}r_{ff}$ \Comment{Project the residual into $\Omega^{h_x,H_y}$}
		\State 4.2) $r_{cf}=P_{ff}^{cf}r_{ff}$ \Comment{Project the residual into $\Omega^{H_x,h_y}$}
		\State 5)\ $\tilde{e}_{ff}=\textbf{InitialGuess}\left(0\cdot b_{ff},r_{fc},r_{cf}\right)$ \Comment{Approximate the error}
		\State 6)\ $\hat{s}=\arg\underset{s\in \mathbb{R}^+}{\min}\|b_{ff}-A_{ff}\left(u^{(k)}+s\tilde{e}_{ff}\right)\|$
		\State 7)\ $u^{(k+1)}=u^{(k)}+\hat{s} \tilde{e}_{ff}$ \Comment{Apply the scaled correction}
		\State 8)\ $k=k+1$
		\EndWhile
		%		\State 9) $tilde{u}_{ff}=u^{(k)}$
		\State\Return $u^{(k)}$
		\EndFunction
	\end{algorithmic}
\end{algorithm}

\noindent After the computation of the initial guess, we compute the residual and iteratively apply corrections using Algorithm~\ref{algorithm:InitialGuess} to provide an initial guess of the error.\\
In details, we compute the residual in step 3) and observe that, as shown in Proposition \ref{prop:zero_res_in_holes} and depicted in Figure \ref{fig:residual}, it is concentrated on the skeleton of the solution. Therefore, in step 4) we project the residual into the anisotropic meshes $\Omega^{h_x,H_y}$ and $\Omega^{H_x,h_y}$ using the projectors $P_{ff}^{fc}$ and $P_{ff}^{cf}$, shown in Figures \ref{fig:res1} and \ref{fig:res1}, respectively. In step 5) we compute an approximate solution $\tilde{e}_{ff}$ of the error equation $A_{ff}e_{ff}=r_{ff}$ through the Inital Guess Algorithm \ref{algorithm:InitialGuess}.
However, we face two challenges in this process: the residual is not smooth enough to allow for the use of Richardson extrapolation, and the different scalings of the two anisotropic error equations $A_{fc}e_{fc}=r_{fc}$ and $A_{cf}e_{cf}=r_{cf}$ may lead to large differences in the norms of the solutions $e_{fc}$ and $e_{cf}$. 
Indeed, we observe that, as reported in Proposition \ref{prop:zero_res_in_holes} and depicted in Figures \ref{fig:res1} and \ref{fig:res2}, the residual surfaces show irregularities which lie close to the coarse mesh points of $\Omega^{H_x,H_y}$ (green), i.e., the cross points of the two anisotropic meshes. Moreover, looking at the $z$-axis in Figures \ref{fig:err1},\ref{fig:err2} we observe a slight difference in scale, which could get larger when increasing the mesh points.\\
Therefore, to address these issues, in step 3) of Algorithm \ref{algorithm:Skeleton}, we randomly choose between $P_{fc}^{cc}u_{fc}$ and $P_{cf}^{cc}u_{cf}$ as $\hat{u}_{cc}$, since $\hat{u}_{cc}$ cannot be extrapolated, while, in step 2) of Algorithm \ref{algorithm:InitialGuess}, we normalize the two solutions before constructing the skeleton. The normalization ensures that the two coarse solutions, which are both sub-samples of a finer solution, have approximately the same norm.\\
In the final step 6) of the main algorithm, we compute an optimal scaling of the update $\tilde{e}_{ff}$ such that it minimizes the residual. This step is needed to correctly scale the solution after the normalization and it is crucial to ensure a non-increasing residual.

\begin{remark}
	The scaling applied in step 7) of the ASDSM Algorithm \ref{algorithm:ASDSM} ensures that the algorithm cannot diverge, but it does not prevent stagnation. Moreover, since the error is computed though Algorithm \ref{algorithm:InitialGuess}, we infer that Proposition \ref{prop:zero_res_in_holes} holds true also for $u^{(k)}$, $\forall k\geq 0$.
\end{remark}

\begin{figure}
	\centering
	\begin{subfigure}{0.48\textwidth}
		\includegraphics[width=1\linewidth]{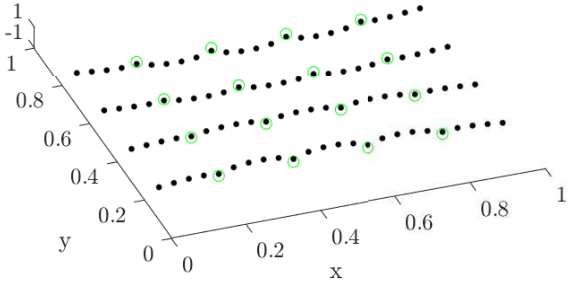}
		\caption{$r_{fc}$: residual $r_{ff}$ projected over the mesh $\Omega^{h_x,H_y}$.}
		\label{fig:res1}
	\end{subfigure}\hfill
	\begin{subfigure}{0.48\textwidth}
		\includegraphics[width=1\linewidth]{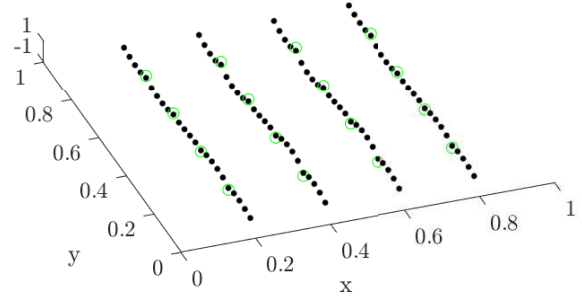}
		\caption{$r_{cf}$: residual $r_{ff}$ projected over the mesh $\Omega^{H_x,h_y}$.}
		\label{fig:res2}
	\end{subfigure}
	\caption{Residual projected over the two anisotropic meshes.}
\end{figure}

\begin{figure}
	\centering
	\begin{subfigure}{0.48\textwidth}
		\includegraphics[width=1\linewidth]{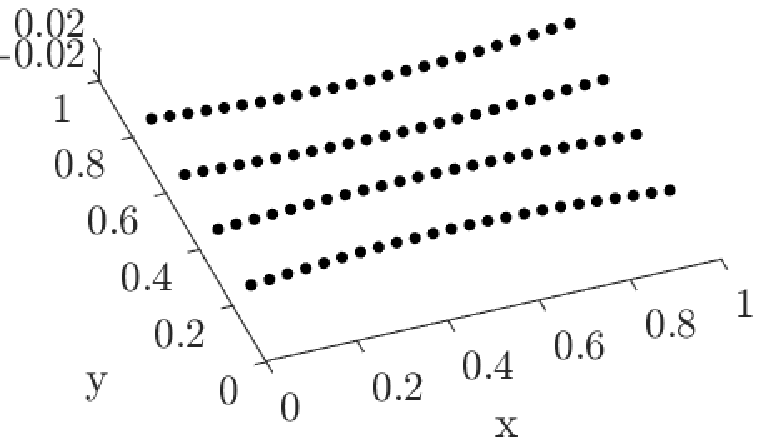}
		\caption{$e_{fc}$: error computed over the mesh $\Omega^{h_x,H_y}$.}
		\label{fig:err1}
	\end{subfigure}\hfill
	\begin{subfigure}{0.48\textwidth}
		\includegraphics[width=1\linewidth]{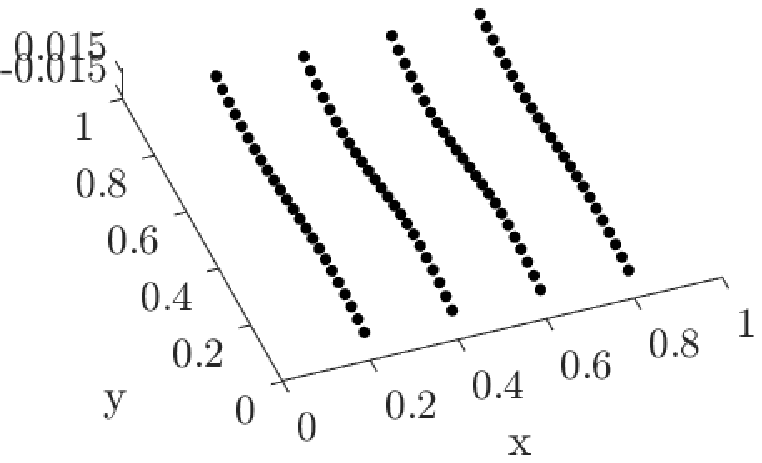}
		\caption{$e_{cf}$: error computed over the mesh $\Omega^{H_x,h_y}$.}
		\label{fig:err2}
	\end{subfigure}
	\caption{Solutions of the two anisotropic error equations.}
\end{figure}

\subsection{Efficient implementation}\label{subsec:Efficient_implementation_and_final_remarks}
In order to efficiently implement ASDSM, the following remarks should be considered.
\begin{itemize}
	\item The sparsity of the residual can be leveraged to decrease the computational cost by restricting the computations to the non-zero elements only;
	\item The computation of $\hat{s}$, which requires to solve an over-determined linear system with one unknown and $O\left( N_c^{(x)}N_f^{(x)}+N_c^{(y)}N_f^{(y)} \right)$ equations, can be made more efficient by considering a random subset of equations, rather than all of them;
	\item When computing the initial guess of the error, the filling process consists in solving many homogeneous PDEs concurrently. In some cases the solution of a homogeneous PDE can be found analytically, and this would essentially allow to eliminate the cost of the filling process;
	\item The inversion of the block matrix in Step 2) of Algorithm \ref{algorithm:Filler} can be done through matrix free approaches \cite{matrixfree};
	\item As a consequence of Proposition \ref{prop:zero_res_in_holes}, smart memory management can be used to reduce the storage requirements for the iterative unknown $u^{(k)}$ by only storing the skeleton and elements close to it, rather than all $N_f^{(x)}N_f^{(y)}$ components. This can bring the storage requirements down to $O\left(N_c^{(x)}N_f^{(y)}+N_c^{(y)}N_f^{(x)}\right)$.
	%According to Proposition \ref{prop:zero_res_in_holes}, the iterative unknown $u^{(k)}$ coincides with the solution $u_{ff}=\left(A_{ff}\right)^{-1}$ if and only if the two skeletons coincide. Since the holes can be easily filled in parallel there is no need to store every component of the unknown $u^{(k)}$. The only important components to store are the ones which allow to compute the non-zero elements of the residual, i.e., the skeleton of $u^{(k)}$ and the elements close to it. A smart memory handling would allow to reduce the storage requirements from $N_{f}^{(x)}N_{f}^{(y)}$ to $O\left(N_{c}^{(x)}N_{f}^{(y)}+N_{c}^{(y)}N_{f}^{(x)}\right)$.
\end{itemize}

\subsection{Extension to different discretizations and PDEs}
In case of higher order PDEs or higher accuracy order of the discretization schemes, the discretization process would yield banded coefficient matrices with bandwidth larger than 3, and a consequent increase in BCs. This can create difficulties in properly filling the holes $\Omega^{h_x,h_y}_{holes}$, as the linear system obtained from the discretization of the PDE over each hole may require more than one BC. 
One potential solution is to use different discretization methods near the edges of the domain. For example, through the Taylor expansion we obtain the following $3$-rd order accurate finite difference discretizations of the second order derivative
\begin{equation*}
	\begin{split}
		u''(x_i)&=\frac{-u_{i-2}+16u_{i-1}-30u_i+16u_{i+1}-u_{i+2}}{12h^2}+O(h^3);\\
		u''(x_i)&=\frac{10u_{i-1}-15u_i-4u_{i+1}+14u_{i+2}-6u_{i+3}+u_{i+4}}{12h^2}+O(h^3);\\
		u''(x_i)&=\frac{u_{i-4}-6u_{i-3}+14u_{i-2}-4u_{i-1}-15u_i+10u_{i+1}}{12h^2}+O(h^3).
	\end{split}
\end{equation*}
Therefore, we could define the linear systems in the holes through the above equalities such that only one BC is required. Then, to ensure that Remark \ref{rem:sub_blocks_of_main_diagonal} holds true, we build the fine coefficient matrix $A_{ff}$ such that its sub-blocks coincide with the linear systems in the holes. Since Remark \ref{rem:sub_blocks_of_main_diagonal} is essential to the proof of Proposition \ref{prop:zero_res_in_holes}, it is then theoretically possible to extend ASDSM towards higher order PDEs or higher order accuracy discretization schemes.

%Suppose that we found a way to add the needed BC without increasing the width of the skeleton, therefore the Skeleton Builder Algorithm \ref{algorithm:Skeleton} remains unchanged. Then the problem transfers to the residual. The perk of having non-zero components of the residual only over the skeleton drops due to Remark \ref{rem:sub_blocks_of_main_diagonal} not holding anymore.\\
%A workaround could be to use a different discretizations near the edges of the domain. For example, through the Taylor expansion we obtain the following $3$-rd order finite difference discretizations of the second order derivative
%\begin{equation*}
%\begin{split}
%u''(x_i)&=\frac{-u_{i-2}+16u_{i-1}-30u_i+16u_{i+1}-u_{i+2}}{12h^2}+O(h^3);\\
%u''(x_i)&=\frac{10u_{i-1}-15u_i-4u_{i+1}+14u_{i+2}-6u_{i+3}+u_{i+4}}{12h^2}+O(h^3);\\
%u''(x_i)&=\frac{u_{i-4}-6u_{i-3}+14u_{i-2}-4u_{i-1}-15u_i+10u_{i+1}}{12h^2}+O(h^3).
%\end{split}
%\end{equation*}
%Therefore, we could define the linear systems in the holes through the above equalities such that only one BC is required. Then, in order for Remark \ref{rem:sub_blocks_of_main_diagonal} to hold, we build the fine coefficient matrix $A_{ff}$ such that its sub-blocks coincide with the linear systems in the holes. Since Remark \ref{rem:sub_blocks_of_main_diagonal} is essential to the proof of Proposition \ref{prop:zero_res_in_holes}, it is then theoretically possible to extend ASDSM towards higher order PDEs or higher order accuracy discretization schemes.

\subsection{Extension to the 3D case}\label{subsec:extension_3D}
In the 3D case, we aim to solve PDE \eqref{eq:genericPDE} over a fine 3D uniform mesh. By extending the notation introduced in Section \ref{subsec:meshes}, we denote such mesh by $\Omega^{h_x,h_y,h_z}$, with $h_x,h_y,h_z$ being the grid widths in each dimension. An optimal way to extend ASDSM would be to discretize the PDE over four different meshes: three strongly anisotropic meshes $\Omega^{H_x,h_y,h_z},\Omega^{h_x,H_y,h_z},\Omega^{h_x,h_y,H_z}$ and one coarse isotropic mesh $\Omega^{H_x,H_y,H_z}$. However, as illustrated in Figure \ref{fig:fig3D_holes1}, the problem lies in the nested meshes. In this case, the cubic empty regions only have of 8 grid points as BC, i.e., the corners of the cube, and there is no solution provided in the faces. Therefore, the obtained skeleton cannot be used as BC for new disjoint subproblems.\\
One potential solution is to consider the denser anisotropic meshes, such as $\Omega^{H_x,h_y,h_z},\Omega^{h_x,H_y,h_z},\Omega^{h_x,h_y,H_z}$. In Figure \ref{fig:fig3D_holes2}, it can be observed that the use of denser meshes results in well-defined BCs for the disjoint subproblems within the holes. However, it is important to note that this approach also leads to a significant increase in computational cost due to the increased density of the meshes.

\begin{figure}
	\centering
	\begin{subfigure}[b]{0.48\textwidth}
		\includegraphics[width=1\linewidth]{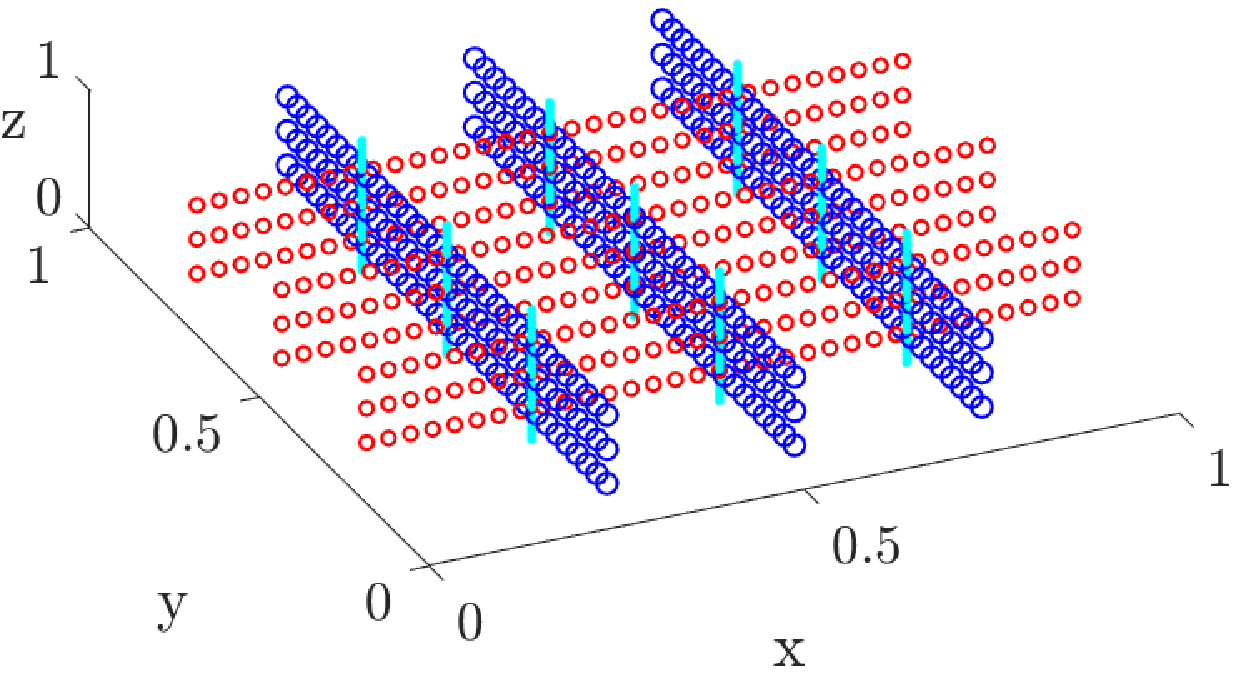}
		\caption{Meshes $\Omega^{h_x,H_y,H_z}$ in red, $\Omega^{H_x,h_y,H_z}$ in blue, $\Omega^{H_x,H_y,h_z}$ in cyan.}
		\label{fig:fig3D_holes1}
	\end{subfigure}
	\hfil
	\begin{subfigure}[b]{0.48\textwidth}
		\includegraphics[width=1\linewidth]{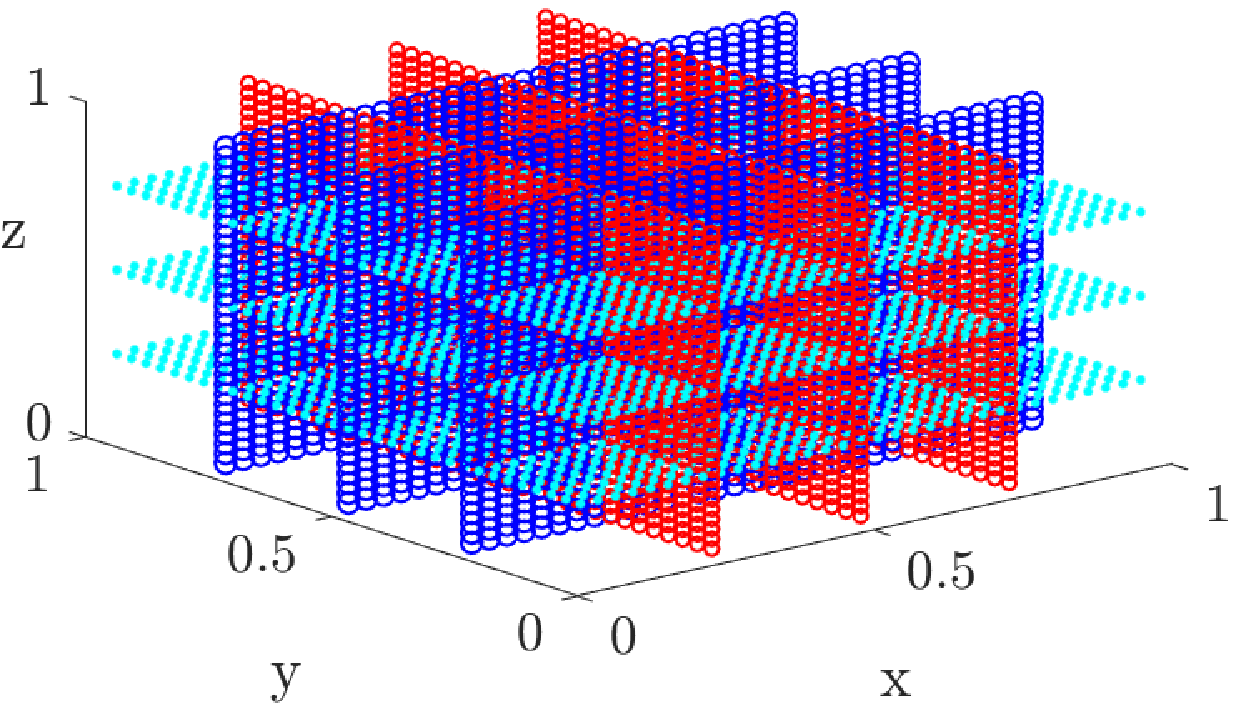}
		\caption{Meshes $\Omega^{H_x,h_y,h_z}$ in red, $\Omega^{h_x,H_y,h_z}$ in blue, $\Omega^{h_x,h_y,H_z}$ in cyan.}
		\label{fig:fig3D_holes2}
	\end{subfigure}
	\caption{Different choice of the anisotropic meshes.}
\end{figure}

In the numerical results section, when addressing the 3D case, we will implement ASDSM by iterating over the meshes $\Omega^{H_x,h_y,h_z},\Omega^{h_x,H_y,h_z},\Omega^{h_x,h_y,H_z}$ with $\Omega^{H_x,H_y,H_z}$ as the coarse intersection mesh. Note that, unlike the 2D case, the mesh $\Omega^{H_x,H_y,H_z}$ does not contain all the intersection points among the previous meshes. For example, $\Omega^{h_x,H_y,H_z}=\Omega^{h_x,h_y,H_z}\cap \Omega^{h_x,H_y,h_z}$. Therefore, in order to ensure a well-defined merging process, we also extrapolate the other intersection points.

\section{Numerical results}\label{sec:Results}
%\emph{Test 1.} Here we show the solution process through Figures step-by-step. Consider $N_H=4$, $N_h=5$ and solve equation \eqref{eq_adv_diffus} with settings 1), $\nu=0$ through algorithm ....
This section presents numerical results evaluating the performance of ASDSM. The implementation is in Matlab 2020b and the tests are run on a machine with AMD Ryzen 5950X processor and 64GB of RAM. \\
The algorithms used in this section are available in \cite{algorithms}.
In the following examples, we set $N_c^{(x)}=N_c^{(y)}=N_c^{(z)}=N_c$ and $N_f^{(x)}=N_f^{(y)}=N_f^{(z)}=N_f$. Moreover, each example will tests ASDSM for solving \eqref{eq:genericPDE} or \eqref{eq:genericPDE_Time} under two scenarios: one where the solution is smooth, and the second with an oscillatory solution. As a results, we use the notation $r_k^{(1)},r_k^{(2)}$ to represent the residual at the $k$-th iteration of ASDSM for solving the PDE under the first and second scenarios, respectively. Note that, when $k=0$, $r_0^{(j)},\ j=1,2$ represents the residual of the initial guess computed through Algorithm \ref{algorithm:InitialGuess}.\\

\textit{Example 1:} Here we test ASDSM for solving the two-dimensional steady state advection-diffusion equation \eqref{eq:genericPDE} with the following settings:
\begin{itemize}
	\item[1)] $\alpha_1=\alpha_2=\beta_1=\beta_2=1$ and functions $s,g$ computed from the exact solution $u(x,y)=\sin(x\pi+y\pi)$;
	\item[2)] $\alpha_1=1+x^2,\ \alpha_2=2+xy,\ \beta_1=2-x,\ \beta_2=1+y$ and functions $s,g$ computed from the exact solution $u(x,y)=\sin(4x\pi+4y\pi)$.
\end{itemize}
We note that settings 1) and 2) differ from the choice of coefficients $\alpha_i,\beta_i$, which in 1) are constants and equal, and in 2) they are taken as positive functions. Additionally, the solution in setting 2) is more oscillatory than in setting 1), making the problem harder to solve, especially in the case of extremely anisotropic meshes, as more points are needed to better approximate the solution.\\
Figures \ref{fig:2D_space_1} and \ref{fig:2D_space_2}, respectively, show the residuals $r_k^{(1)},r_k^{(2)}$ at the iteration $k=0,...,20$ for the two settings described above. The dotted, dashed and circled lines represent $N_c=5,10,20$ respectively, while the red, green, blue and magenta lines represent $N_f=400,800,1600,3200$. In each case, the size of the linear system being solved is $N_f\times N_f$, therefore, different type of lines with the same color represent the residual of ASDSM while solving the same linear system. This is because $N_c$ is only needed in ASDSM to define the size of the anisotropic linear systems. Hence, larger $N_c$ leads to a more computationally expensive iteration of ASDSM.

\begin{figure}
	\centering
	\begin{subfigure}[b]{0.48\textwidth}
		\includegraphics[width=1\linewidth]{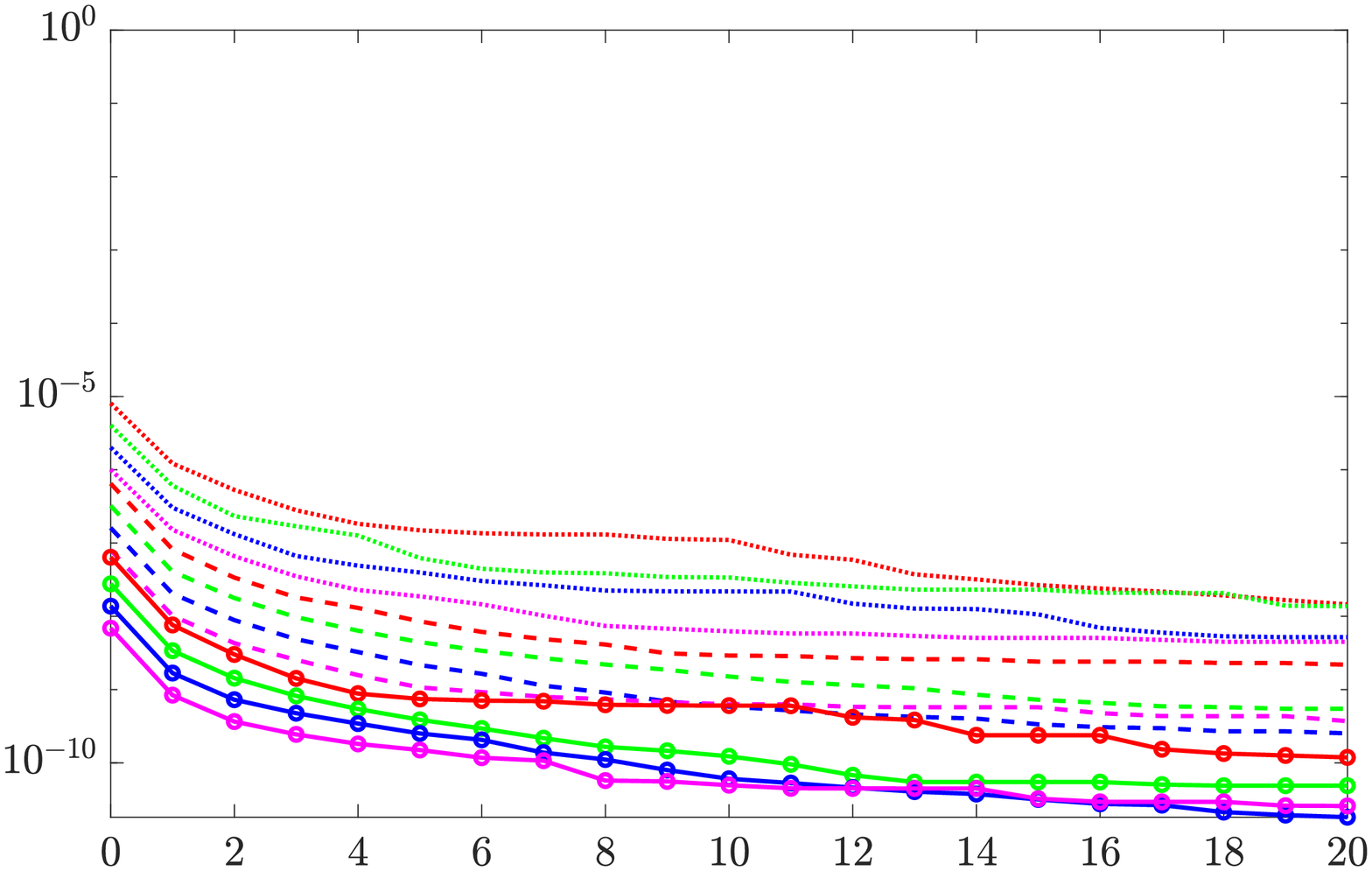}
		\caption{Setting 1)}
		\label{fig:2D_space_1}
	\end{subfigure}
	\hfil
	\begin{subfigure}[b]{0.48\textwidth}
		\includegraphics[width=1\linewidth]{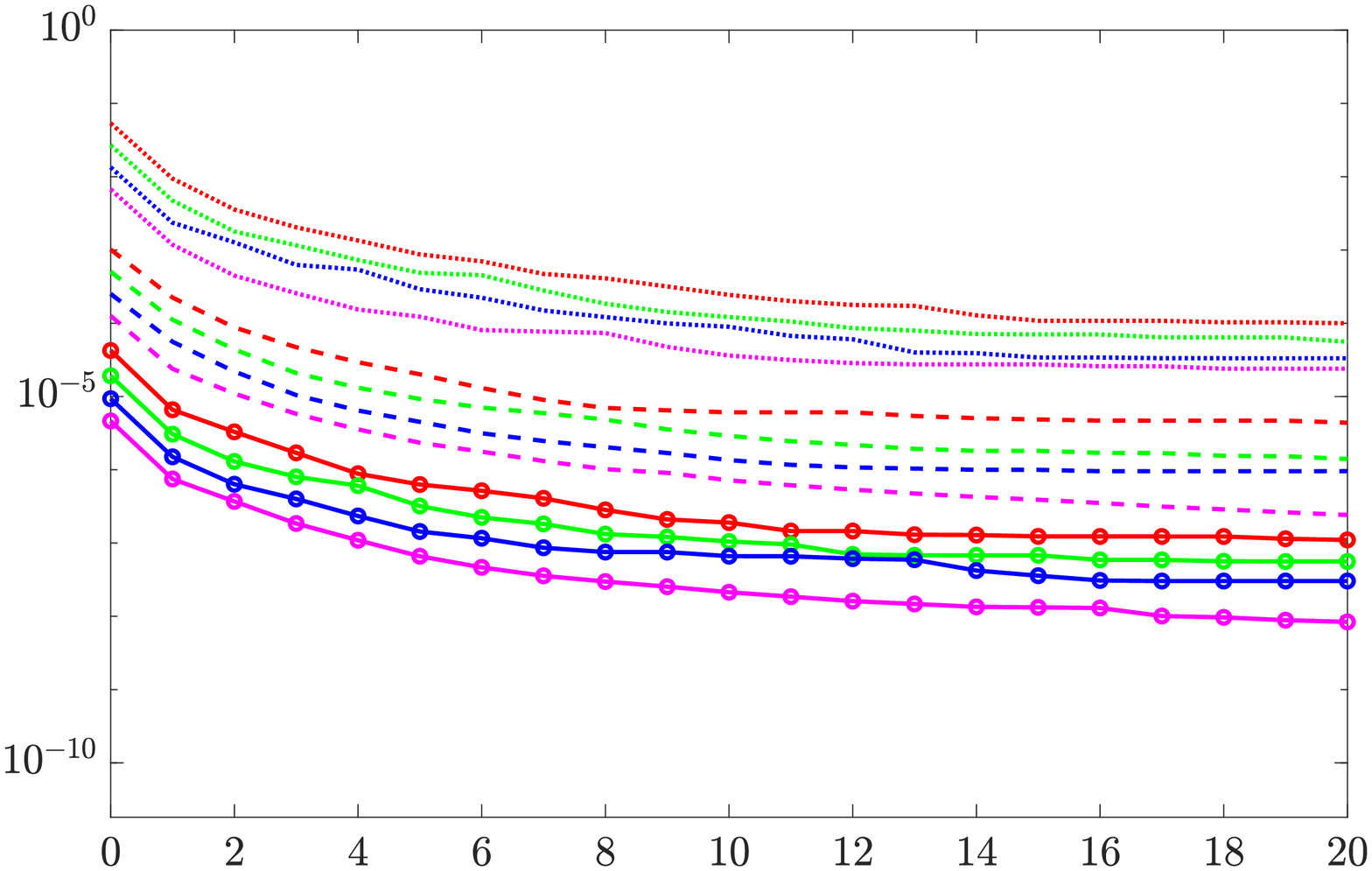}
		\caption{Setting 2)}
		\label{fig:2D_space_2}
	\end{subfigure}
	\caption{Residuals $r_k^{(1)},r_k^{(2)}$ at each iteration of ASDSM applied to equation~\eqref{eq:genericPDE} for the two settings. The dotted, dashed and circled lines represent $N_c=5,10,20$ respectively, while the red, green, blue and magenta lines represent $N_f=400,800,1600,3200$.}
	\label{fig:2D_space}
\end{figure}

In Figures \ref{fig:2D_space_1}, where setting 1) is considered, it can be observed that the Richardson extrapolation inside the initial guess algorithm leads to a strong initial boost, resulting in a very small $r_0^{(1)}$ for every combination of $N_c,N_f$. Then, the first iteration of ASDSM further strongly reduce the residual by approximately one order of magnitude. However, after the first iteration, the reduction factor decreases and seems to converge to $1$ as $k$ increases, leading to stagnation. One possible cause for this is the low accuracy of the anisotropic linear system. As ASDSM gets closer to the solution, the initial guess of the error becomes less and less accurate. 
%The problem may be caused by the low accuracy of the anisotropic linear system. The initial guess of the error becomes less and less accurate, the closer ASDSM gets to the solution.
Indeed, we note that by fixing $N_f$ and increasing $N_c$, i.e., reducing the anisotropy and increasing accuracy at the same time, the stagnation point lowers. Another cause may be related to the smoothness of the residual (see Example 2 for further details).\\
Furthermore, when fixing $N_c$ and increasing $N_f$, which increases the anisotropy, the stagnation point still appears to decrease. This suggests that increasing the size of main linear system, and consequently the accuracy of its solution, improves the performance of ASDSM.
% Therefore, by increasing the size of the linear system, and consequently the accuracy of its solution, the performance of ASDSM seems to improve.\\

When considering setting 2), in Figure \ref{fig:2D_space_2} we note a similar behavior of $r_k^{(2)}$ with respect to $r_k^{(1)}$. The main difference lies in the starting point at $k=0$. In this case, of an oscillatory solution, the extrapolated cross-points are less accurate than in the previous case. Therefore, the initial guess algorithm yields a less accurate solution. Another loss of accuracy may come from the use of variable coefficients $\alpha_i,\beta_i,\ i=1,2$ in place of constant ones used in setting 1). Indeed, when $N_c=5$ the residual $r_0^{(1)}$ is almost 4 orders of magnitude smaller than $r_0^{(2)}$, and 3 orders of magnitude when $N_c=10,20$.\\

\begin{figure}
	\centering
	\begin{subfigure}[b]{0.48\textwidth}
		\includegraphics[width=1\linewidth]{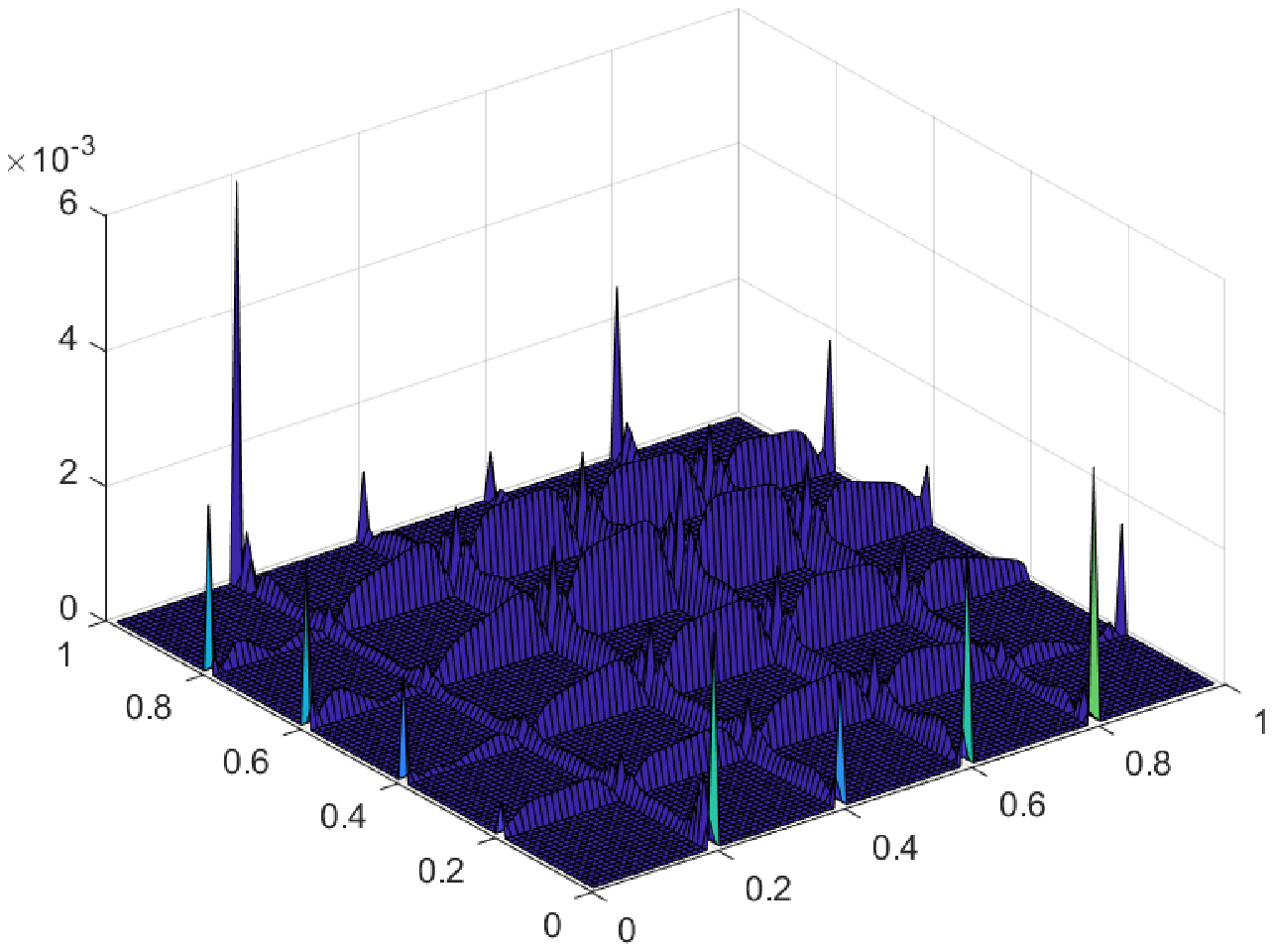}
		\caption{Steady state PDE 1)}
		\label{fig:irreg_res_space}
	\end{subfigure}
	\hfil
	\begin{subfigure}[b]{0.48\textwidth}
		\includegraphics[width=1\linewidth]{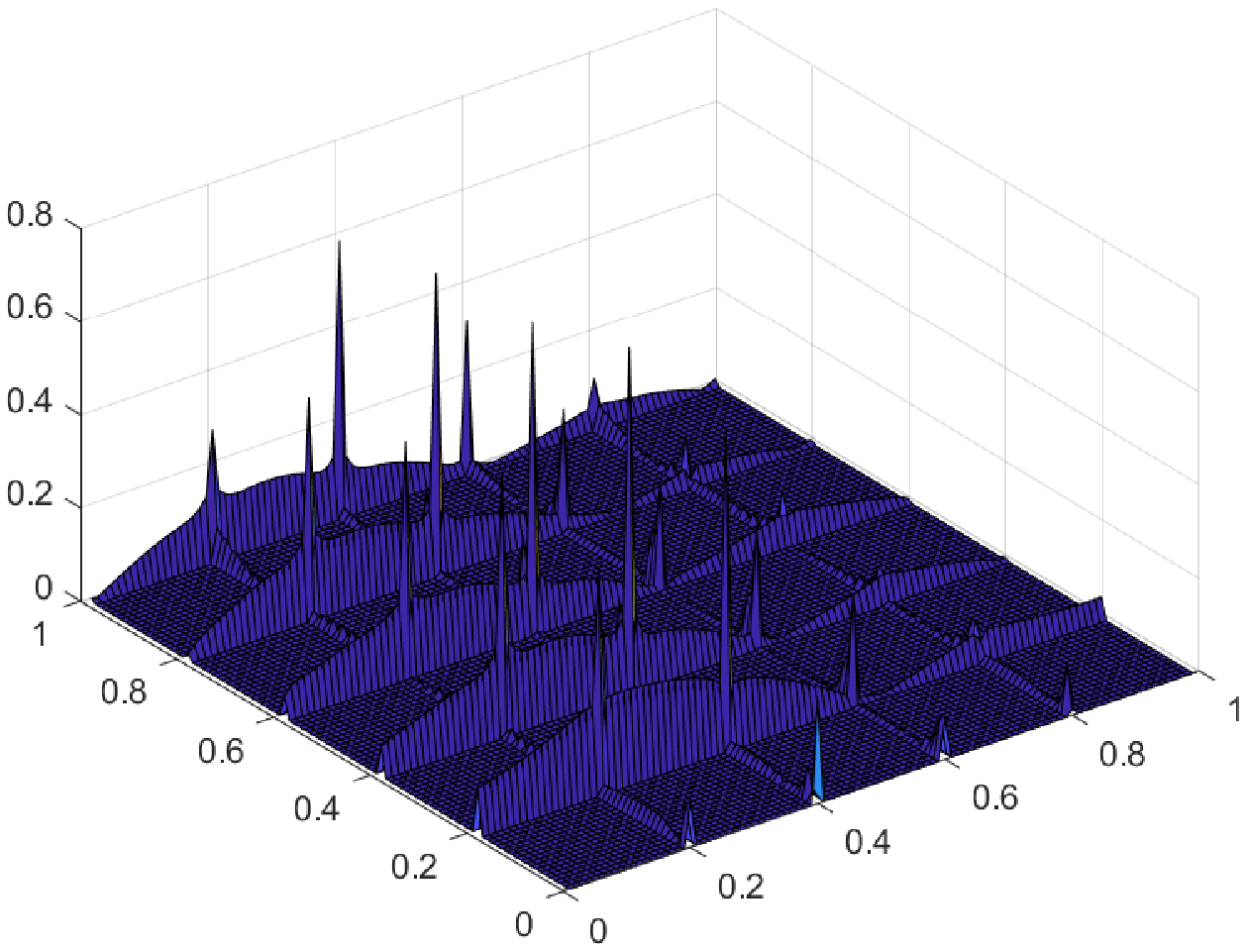}
		\caption{Time-dependent PDE 2)}
		\label{fig:irreg_res_time}
	\end{subfigure}
	\caption{Reshaped absolute residual after 20 iterations of ASDSM}
	\label{fig:irreg_res}
\end{figure}

\textit{Example 2:}
In this example we show an important drawback of ASDSM, which consists in making the residual highly irregular at the coarse points. Consider the following two toy problems:
\begin{itemize}
	\item[1)] 2-dimensional space steady state advection diffusion equation \eqref{eq:genericPDE} with $\alpha_i=\beta_i=1,\ i=1,2$ and $s,g$ computed from the exact solution $u(x,y)=\sin(x\pi+y\pi)$;
	\item[2)] 1-dimensional space time-dependent advection diffusion equation \eqref{eq:genericPDE_Time} with $\alpha_1=\beta_1=1$ and $s,g$ computed from the exact solution $u(x,t)=\sin(x\pi+t\pi)$.
\end{itemize} 
In Figure \ref{fig:irreg_res} we show the reshaped residual obtained after iterating ASDSM 20 times, i.e., when ASDSM stagnates. In both cases, strong irregularities can be observed at the coarse points. In the time-dependent case, less severe irregularities are also observed along the mesh dense in time $\Omega^{H_x,h_t}$. This results in highly irregular projected residuals in steps 4.1) and 4.2) of the ASDSM Algorithm \ref{algorithm:ASDSM}, leading to inaccurate solutions in steps 1.1) and 1.2) of Algorithm \ref{algorithm:InitialGuess} which in turn causes the stagnation.\\
Additional tests, not presented in this paper, have shown that an artificial smoothing of the residual reduces the stagnation point.

\begin{figure}
	\centering
	\begin{subfigure}[b]{0.48\textwidth}
		\includegraphics[width=1\linewidth]{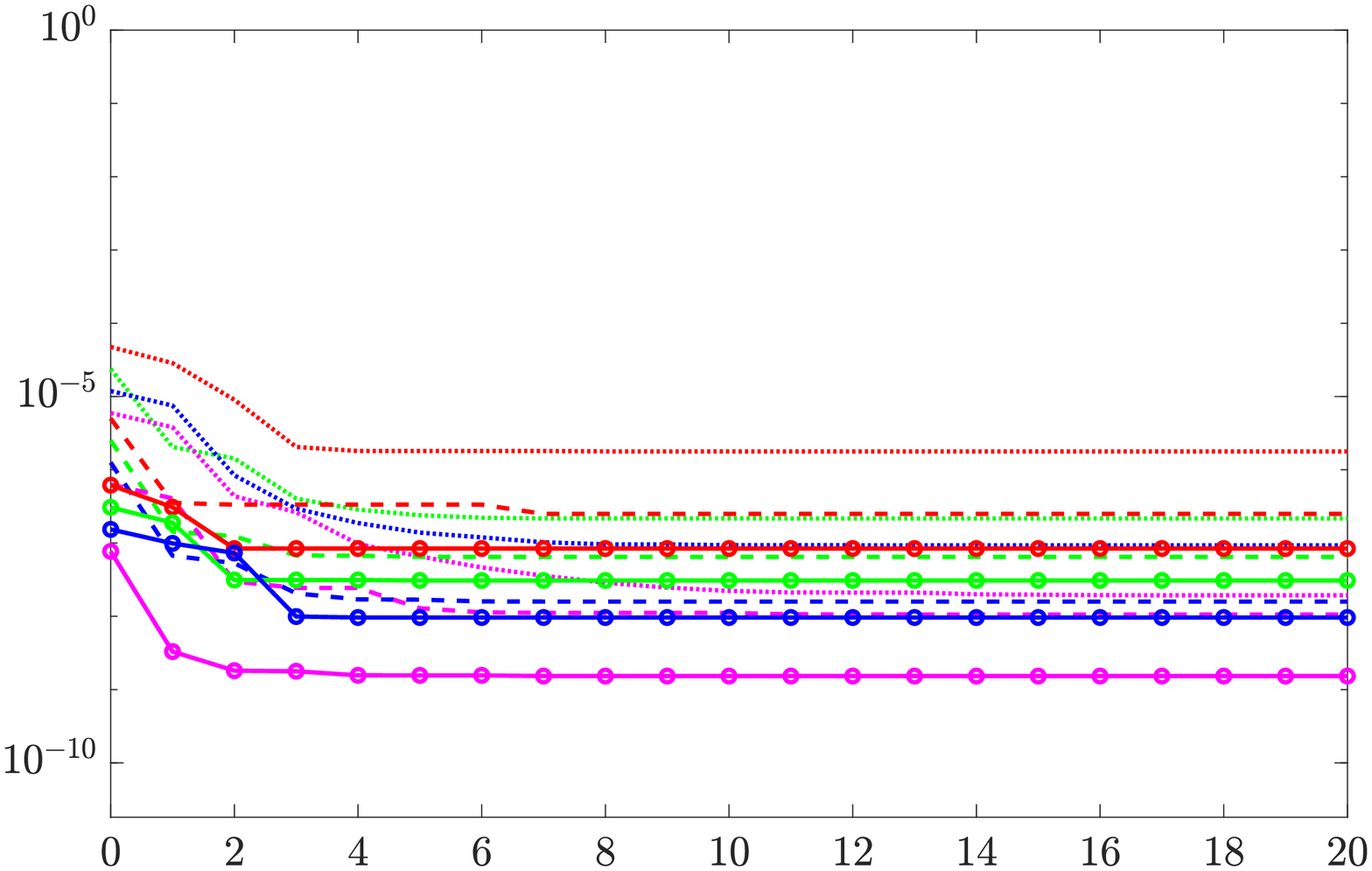}
		\caption{Setting 1)}
		\label{fig:2D_time_1}
	\end{subfigure}
	\hfil
	\begin{subfigure}[b]{0.48\textwidth}
		\includegraphics[width=1\linewidth]{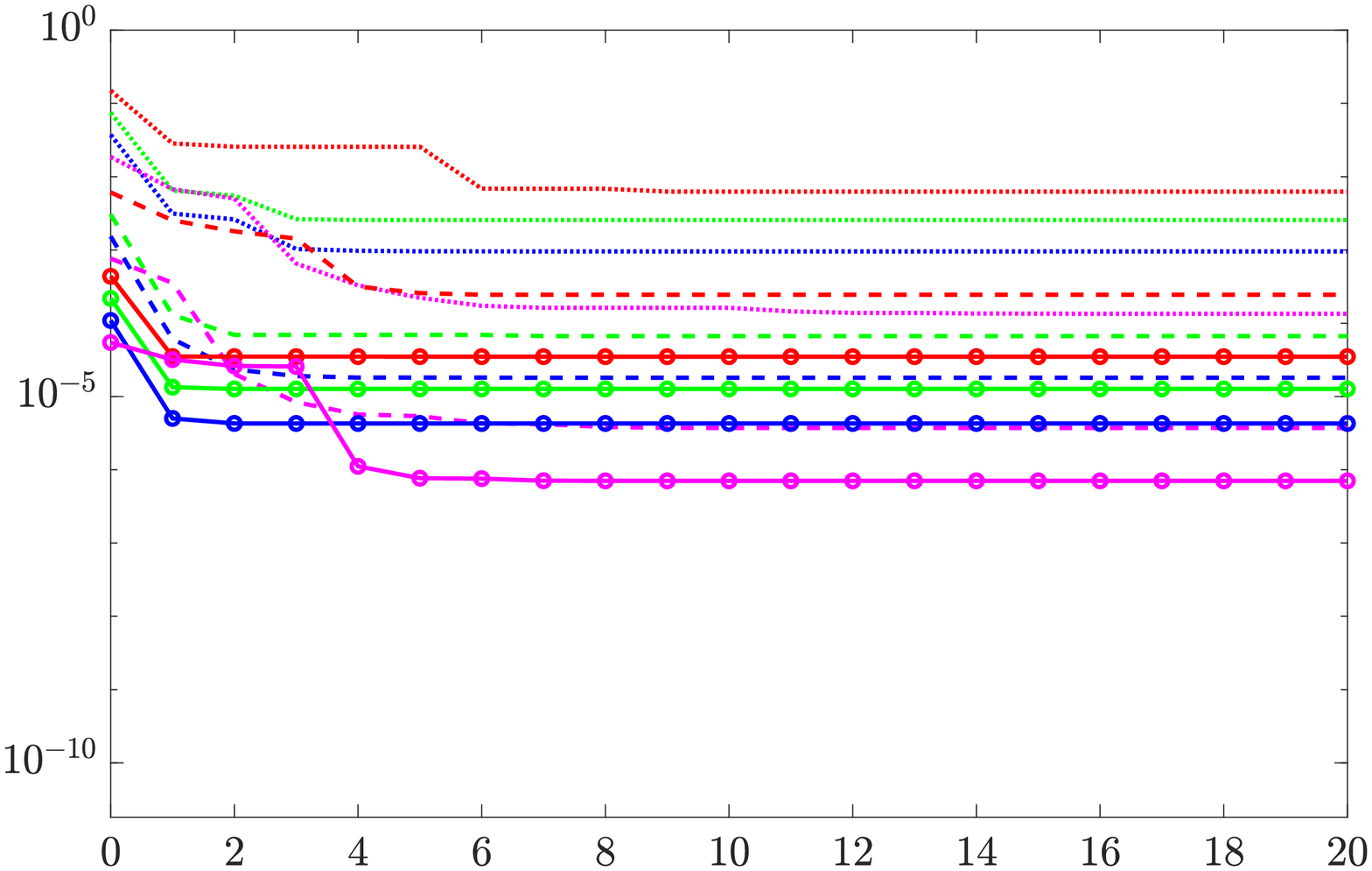}
		\caption{Setting 2)}
		\label{fig:2D_time_2}
	\end{subfigure}
	\caption{Residuals $r_k^{(1)},r_k^{(2)}$ at each iteration of ASDSM applied to equation~\eqref{eq:genericPDE_Time} for the two settings. The dotted, dashed and circled lines represent $N_c=5,10,20$ respectively, while the red, green, blue and magenta lines represent $N_f=400,800,1600,3200$.}
	\label{fig:2D_time}
\end{figure}

\textit{Example 3:} Here we test ASDSM for solving the one-dimensional space time-dependent advection-diffusion equation \eqref{eq:genericPDE_Time} with the following settings:
\begin{itemize}
	\item[1)] $\alpha_1=\beta_1=1$ and functions $s,g$ computed from the exact solution $u(x,t)=\sin(x\pi+t\pi)$;
	\item[2)] $\alpha_1=1+x^2,\ \beta_1=2-x$ and functions $s,g$ computed from the exact solution $u(x,t)=\sin(4x\pi+4t\pi)$.
\end{itemize}
Note that the two settings, which involve a smooth and an oscillatory solution, are the same as in Example 1, but with the variable $t$ replacing $y$.\\
With the same notation as in Example 1, Figures \ref{fig:2D_time_1} and \ref{fig:2D_time_2} show the absolute residuals $\abs{r_k^{(1)}},\abs{r_k^{(2)}}$ at the $k$-th iteration, respectively. A comparison between Figures \ref{fig:2D_time} and \ref{fig:2D_space} shows that, in general, the same observations made in Example 1 apply to this time-dependent case as well. However, the main difference is in the trend of the residual, which in Figure \ref{fig:2D_time} is shown to be less smooth than in Figure \ref{fig:2D_space}. The reason for such behavior lies in step 3) of the Skeleton builder Algorithm \ref{algorithm:Skeleton}, where the cross-points of the error surface cannot be extrapolated, and a random solution is taken as ``the most accurate". In this case, the solution $e_{H_x,h_t}$ computed over the dense in time mesh $\Omega^{H_x,h_t}$ is more accurate than $e_{h_x,H_t}$ computed over the dense in space mesh $\Omega^{h_x,H_t}$. Therefore, when the randomly chosen cross-points are the ones of error $e_{H_x,h_t}$, the residual strongly reduces. However, this large reduction in residual only seems to happen a couple of times before stagnation occurs for the same reasons reported in Example 1.\\

\textit{Example 4:} Here we test ASDSM for solving three-dimensional steady state advection-diffusion equation \eqref{eq:genericPDE} with the following settings:
\begin{itemize}
	\item[1)] $\alpha_i=\beta_i=1,\ i=1,2,3$ and functions $s,g$ are computed from the exact solution $u(x,y,z)=\sin(x\pi+y\pi+z\pi)$;
	%	\item[2)] $\alpha_i=\beta_i=i,\ i=1,2,3$ and functions $s,g$ computed from the exact solution $u(x,t)=sin(4x\pi+4y\pi+4z\pi)$.	
	\item[2)] $\alpha_1=1+x^2,\ \alpha_2=2+xy,\ \alpha_3=3-xyz,\ \beta_1=2-x,\ \beta_2=1+y,\ \beta_3=2-x+yz$ and functions $s,g$ are computed from the exact solution $u(x,t)=\sin(4x\pi+4y\pi+4z\pi)$.
\end{itemize}
The smooth and oscillatory solutions of settings 1) and 2) are extensions of the settings in Exercise 1, with the addition of the time variable $t$. In Figures \ref{fig:3D1},\ref{fig:3D2} we respectively show the residuals $r_k^{(1)},r_k^{(2)}$ at the $k$-th iteration of ASDSM applied to PDE \eqref{eq:genericPDE} with settings 1) and 2). To reduce the computational time of the tests, the results for $N_c=20$, previously reported as the circled line, are not provided and the study is restricted to $N_f\in\lbrace 70,150,300\rbrace$, represented by colors red, green and blue, respectively.\\
From Figure \ref{fig:3D} we can observe that even though the 3-dimensional case is more complex from an algorithmic point of view, and also involves more cross-points, the behavior of the residual is similar to that of Example 1. Therefore, the same observations apply and it is expected that similar results will be obtained for higher values of $N_c,N_f$.

\begin{figure}
	\centering
	\begin{subfigure}[b]{0.48\textwidth}
		\includegraphics[width=1\linewidth]{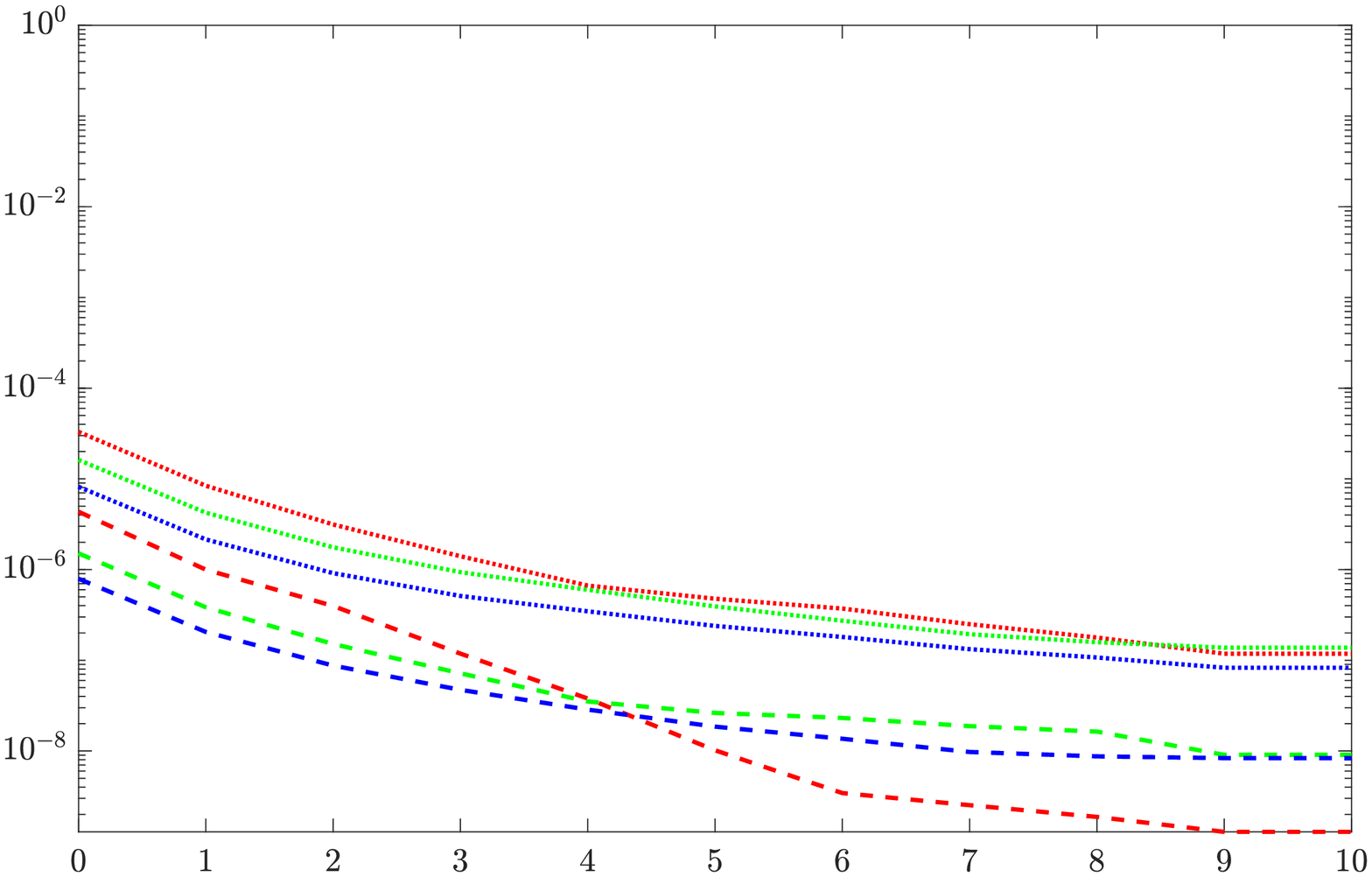}
		\caption{Setting 1)}
		\label{fig:3D1}
	\end{subfigure}
	\hfil
	\begin{subfigure}[b]{0.48\textwidth}
		\includegraphics[width=1\linewidth]{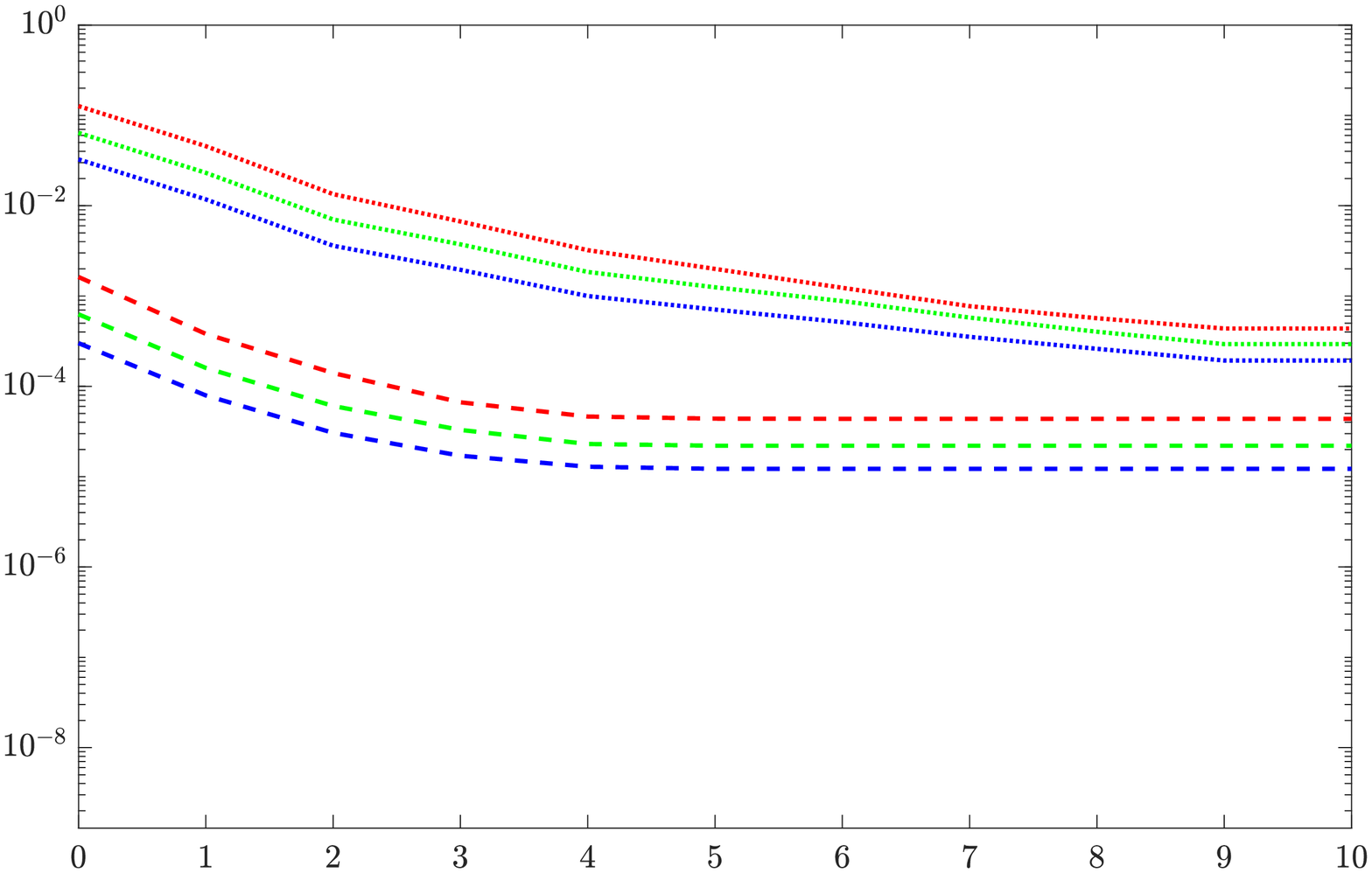}
		\caption{Setting 2)}
		\label{fig:3D2}
	\end{subfigure}
	\caption{Residuals $r_k^{(1)},r_k^{(2)}$ at each iteration of ASDSM applied to equation~\eqref{eq:genericPDE} for the two settings. The dotted, dashed lines represent $N_c=5,10$ respectively, while the red, green, blue and magenta lines represent $N_f=70,150,300$.}
	\label{fig:3D}
\end{figure}

\section{Conclusions}\label{sec:Conclusions}
The ASDSM Algorithm opens a new window on methods based on structured sparse residuals. These preliminary results show that, by decomposing the problem into disjoint sub-problems and merging the solutions, we are able to effectively reduce the dimensionality of the problem by $1$ and improve computational efficiency. 
%it is possible to decompose a $d$-dimensional problem into $d$ $(d-1)$-dimensional disjoint sub-problems, and then merge together the solutions. 
While further research is needed to address issues such as convergence and optimal choice of coarse points, the ASDSM algorithm has the potential to be a valuable tool for tackling large linear systems as solver or by providing a good initial guess for other iterative solvers, e.g., Domain Decomposition methods. It may also be used to provide a fast solution for shooting methods, which do not always require a high accuracy.

In future work, we will explore alternatives to overcome the limitations of the current algorithm, as well as examine its potential application to non-linear problems and unstructured meshes.

\section*{Acknowledgments} 
This research was supported by the EuroHPC TIME-X project n. 955701.

%\begin{acknowledgements}
%The first, third, and fourth authors are members of the INdAM research group GNCS. 
%The second author is member of
%This research was supported by the GNCS-INDAM (Italy) and the Swiss National Science Foundation SNF via the
%projects Stress-Based Methods for Variational Inequalities in Solid Mechanics n. 186407 and ExaSolvers n. 162199.
%\end{acknowledgements}




% end of file template.tex

\end{document}